\documentclass[openany,reqno,12pt,a4paper]{book}

\usepackage{amsmath,amsthm,amssymb,amsfonts,enumerate,graphicx,color,pstricks}

\usepackage{etoolbox}
\makeatletter
\patchcmd{\@makechapterhead}{50\p@}{0pt}{}{}
\patchcmd{\@makeschapterhead}{50\p@}{0pt}{}{}
\makeatother

\newtheorem{theo+}           {Theorem}      [section]

\newtheorem{prop+}  [theo+]  {Proposition}
\newtheorem{coro+}  [theo+]  {Corollary}
\newtheorem{lemm+}  [theo+]  {Lemma}
\newtheorem{defi+}  [theo+]  {Definition}
\newtheorem{conj+}  [theo+]  {Conjecture}

\theoremstyle{definition}
\newtheorem{rema+}  [theo+]  {Remark}
\newtheorem{prob+}   {Exercise}[section]
\newtheorem{exam+}  [theo+]  {Example}

\newenvironment{theorem}{\begin{theo+}}{\end{theo+}}
\newenvironment{proposition}{\begin{prop+}}{\end{prop+}}
\newenvironment{corollary}{\begin{coro+}}{\end{coro+}}
\newenvironment{lemma}{\begin{lemm+}}{\end{lemm+}}

\newenvironment{problem}  {\begin{prob+}}{\end{prob+}}

\newcommand{\ti}{\mathrm i}
\DeclareMathOperator*{\Res}{Res}
\newcommand{\sn}{\operatorname{sn}}
\newcommand{\cn}{\operatorname{cn}}
\newcommand{\dn}{\operatorname{dn}}

\begin{document}

\title
{Elliptic Hypergeometric Functions\\
${}$\\
Lectures at OPSF-S6\\
College Park, Maryland, 11-15 July 2016}
\author{Hjalmar Rosengren\\
Department of Mathematical Sciences\\
Chalmers University of Technology and University of Gothenburg}
\date{}

\maketitle        
\tableofcontents

\chapter*{Preface}
\addcontentsline{toc}{chapter}{Preface}

Various physical models and mathematical objects come in three levels: rational, trigonometric and elliptic.
A classical example is Weierstrass's theorem, which states that a meromorphic one-variable function $f$  satisfying an algebraic addition theorem, that is,
\begin{equation}\label{wad} P\big(f(w),f(z),f(w+z)\big)\equiv 0,\end{equation}
for some polynomial $P$, is either rational, trigonometric or elliptic. Here, trigonometric means that $f(z)=g(q^z)$ for some
rational function $g$, where $q$ is a fixed number. Writing $q=e^{2\ti\pi\eta}$, one may express $f$ in terms of trigonometric functions.
Elliptic means that $f$ has two independent periods.
If we let one of these periods tend to infinity, elliptic solutions of \eqref{wad} degenerate to trigonometric ones. Letting the remaining period
tend to infinity, we recover rational solutions.

Further examples of the hierarchy rational -- trigonometric -- elliptic are abundant in the context
of classical and quantum integrable systems. 
Integrability is closely related to exact solvability, which means that some physically interesting quantities can be computed 
exactly.  The meaning of the word ``exactly'' is  loose but the answer may, for instance, involve hypergeometric functions.
From this perspective, it is not surprising that there is a hierarchy of rational, trigonometric and elliptic hypergeometric functions.
What is perhaps more surprising is that only the first two levels were known classically, with fundamental contributions by mathematicians such as
Euler, Gauss, Cauchy and Heine. Elliptic hypergeometric functions appeared much later, first in the work of Date et al.\ \cite{d} from 1988 and 
more explicitly in the 1997 paper \cite{ft} by Frenkel and Turaev. 
These authors only consider elliptic hypergeometric functions defined by finite sums. An important step forward was taken by
Spiridonov \cite{sp1,scb} who introduced elliptic hypergeometric integrals.
Since the turn of the millenium, development has been  rapid;
in June 2017, the on-line bibliography  \cite{rb}
contained 186 entries. A new wave of interest from physicists was initiated by Dolan and Osborn \cite{do}, who found that elliptic hypergeometric integrals appear in the context of four-dimensional quantum field theories.

The purpose of the present notes is to give an elementary introduction to elliptic hypergeometric functions.
They were written for the summer school OPSF-S6 on orthogonal polynomials and special functions, but I hope that they can be useful also
in other contexts.  I focus on motivating  and exemplifying the main ideas, rather than giving a comprehensive survey of relevant results. The required background knowledge is modest
and should be covered by a first course in complex analysis and some basic notions from linear and abstract algebra.
Previous acquaintance with special functions  will make the material
easier to digest, but is not required.

Chapter \ref{esc} provides a brief introduction to elliptic functions. 
The presentation  may seem idiosyncratic to some readers, but I believe it is
 close to the thinking of many contemporary researchers on elliptic integrable systems.
Most textbooks follow a combination of Jacobi's and Weierstrass's approaches, which hides the elegance of the theory
by cumbersome and (at least for our purposes)  useless notation. 
My philosophy has been to completely avoid notation for specific elliptic functions. Another point is to consistently work with expressions for
theta functions and elliptic functions as infinite products rather than series. In my opinion, this is more natural and often simplifies the theory.

The main part of the text is Chapter \ref{ehc}, where I give an introduction to elliptic hypergeometric sums and integrals. Although a large
part of the literature deals with multivariable functions, I have decided to restrict to the one-variable theory.
The main results are then the Frenkel--Turaev summation and Spiridonov's elliptic beta integral evaluation. To give an 
indication of further results, I also present a quadratic summation and a Karlsson--Minton-type summation. Finally, in Chapter~\ref{slc} I briefly 
explain the historical origin of elliptic hypergeometric functions in the context of solvable lattice models.
In particular, I give a new proof of the fact that fused Boltzmann weights for Baxter's elliptic solid-on-solid model can be expressed as
elliptic hypergeometric sums.

As they are based on a one-week course, the present lecture notes are very limited in scope. Let me provide some
suggestions for further reading.  More extensive introductions to elliptic hypergeometric functions are given in  Chapter~11 of the textbook \cite{gr} and in the survey \cite{ss}. For multivariable elliptic hypergeometric sums, a natural starting point would be \cite{rr}, where some of the more accessible results
are derived in an elementary manner much in the spirit of the present notes. Rains \cite{r1,r2} goes much further, introducing elliptic extensions of Okounkov's interpolation polynomials and Koornwinder--Macdonald polynomials. 
A succinct but rather comprehensive overview of multivariable elliptic hypergeometric functions is given in \cite{rw}.
The reader interested in relations to the Sklyanin algebra and other elliptic quantum groups could start with \cite{ds,knr,rsk}.
One  emerging research area is ``elliptic combinatorics'', where combinatorial objects are dressed with elliptic weight functions, see e.g.\ 
\cite{be,se}. In mathematical physics, there is much on-going activity on relations to four-dimensional supersymmetric quantum field theories. 
For someone with my own mathematical background the literature is hard to get into, but I recommend the reader to have a look at
 \cite{sv1,sv2}, where many ($>100$)  intriguing new  integral identities are conjectured. A related topic is connections between elliptic hypergeometric integrals and 
 two-dimensional lattice models with continuous spin, see \cite{bs,ses} for a start. Some of these recent applications in physics are briefly surveyed in \cite{spn}. Naturally, the above selection is  biased by my own taste and interests. A more complete list of references can be found in \cite{rb}.

{\bf Acknowledgements:} I would like to thank the organizers of OPSF-S6 for inviting me. I am grateful Gaurav Bhatnagar and Linnea Hietala, as well as the anonymous referees, for many useful comments on the manuscript. Finally, I  thank all the students who followed the lectures and contributed to the course.

\chapter{Elliptic functions}
\label{esc}

\section{Definitions}\label{eds}

The classical definition of
an \emph{elliptic function} is a meromorphic function $f$ on $\mathbb C$ with two  periods $\eta$ and $\tau$, that is,
\begin{equation}\label{dp}f(z+\eta)=f(z+\tau)=f(z),\qquad z\in\mathbb C. \end{equation}
To avoid trivialities, one assumes that $\eta$ and $\tau$ are non-zero and $\tau/\eta\notin\mathbb R$.
 Possibly interchanging $\eta$ and $\tau$, we may assume that $\operatorname{Im}(\tau/\eta)>0$.
 Finally, after the change of variables $z\mapsto \eta z$, we may take $\eta=1$.
Thus, it is enough to consider meromorphic functions satisfying
\begin{equation}\label{ap}f(z+1)=f(z+\tau)=f(z),\qquad z\in\mathbb C,\end{equation}
where $\operatorname{Im}(\tau)>0$. 

This ``additive'' definition goes back to Abel's memoir from 1827 and previous unpublished work of Gauss.
We will mostly work with an equivalent, ``multiplicative'', definition.
Note first that, if $f$ is a meromorphic function satisfying $f(z+1)=f(z)$, then we
can introduce a new function $g$ by $f(z)=g(e^{2\ti\pi z})$. Then, $g$ is meromorphic on the punctured plane $\mathbb C^\ast=\mathbb C\setminus\{0\}$.
The periodicity $f(z+\tau)=f(z)$ is  equivalent to $g(px)=g(x)$, where $p=e^{2\ti\pi \tau}$. Thus, we can alternatively define an \emph{elliptic function}
as a meromorphic function $g$ on $\mathbb C^\ast$ such that $g(px)=g(x)$ for all $x$, where the \emph{period} $p$ satisfies $0<|p|<1$.  
We will distinguish the
two definitions by using the terms \emph{additively elliptic} and \emph{multiplicatively elliptic}, respectively.

One can also give a coordinate-free definition of an elliptic function 
as an analytic function from a complex torus (compact Riemann surface of genus one)
to a complex sphere  (compact Riemann surface of genus zero). Our two definitions then correspond to two distinct
choices of a complex coordinate on the torus. For the additive definition, we realize the torus as a parallelogram with opposite edges
identified, for the multiplicative definition as an annulus with the inner and outer boundary circles identified.

\begin{problem}
What can you say about functions satisfying \eqref{dp} when $\tau/\eta\in\mathbb R$?
\end{problem}
\begin{problem}
What can you say about meromorphic functions with \emph{three} additive periods?
\end{problem}

\section{Theta functions}

We want to think of elliptic functions as analogues of rational functions, which can be factored as
\begin{equation}\label{rf}f(z)=C\frac{(z-a_1)\dotsm (z-a_m)}{(z-b_1)\dotsm(z-b_n)}.\end{equation}
We will see that there is an analogous result for elliptic functions, where the building blocks (analogues of  first degree polynomials) 
are known as \emph{theta functions}.

How can we find an elliptic analogue of \eqref{rf}? We expect that the individual  factors on the right should correspond to zeroes and poles of
$f$. In particular,  the analogue  of the building block $z-0$ should vanish at $z=0$. If we want to construct a solution to \eqref{ap}, it is natural
to assume that it vanishes at the whole lattice $ \mathbb Z+\tau\mathbb Z$. In multiplicative language (that is, writing $x=e^{2\ti\pi z}$, $p=e^{2\ti\pi\tau}$),
we are looking for a function vanishing for $x\in p^{\mathbb Z}$. A naive way to construct such a function would be as an infinite product
$$\dotsm(x-p^{-2})(x-p^{-1})(x-1)(x-p)(x-p^2)\dotsm=\prod_{k=-\infty}^\infty (x-p^k). $$
However, this product diverges. For convergence, the factors should  tend to $1$ as $k\rightarrow\pm \infty$, but in fact they tend to $x$ as $k\rightarrow \infty$
and behave as $-p^k$ when $k\rightarrow -\infty$. It is therefore natural to normalize the product by dividing factors with large $k$ by $x$ and factors  with large negative $k$
by $-p^k$. The details of how this is done are not important; we will  make these substitutions for $k>0$ and $k\leq 0$, respectively, 
 denoting the resulting function  $\theta(x;p)$. That is,
$$\theta(x;p)=\prod_{k=-\infty}^0\left(1-\frac{x}{p^k}\right)\prod_{k=1}^\infty \left(1-\frac{p^k}{x}\right)=\prod_{k=0}^\infty(1-xp^k) \left(1-\frac{p^{k+1}}{x}\right).$$
  Equivalently, in the standard notation
  $$(a;p)_\infty=\prod_{k=0}^\infty(1-ap^k),\qquad (a_1,\dots,a_m;p)_\infty=(a_1;p)_\infty\dotsm (a_m;p)_\infty, $$
we have\footnote{The reader who is more familiar with the four classical Jacobi theta functions should have a look at Exercise \ref{jcte}.}
$$\theta(x;p)=(x,p/x;p)_\infty. $$

It will be convenient to use the shorthand notation
$$\theta(a_1,\dots,a_m;p)=\theta(a_1;p)\dotsm\theta(a_m;p)$$
as well as
\begin{equation}\label{tpm}\theta(ax^\pm;p)=\theta(ax;p)\theta(a/x;p).\end{equation}
Note that the \emph{trigonometric limit} $\tau\rightarrow\ti\infty$ corresponds to $p\rightarrow 0$. Then, our theta function reduces to the first degree polynomial $\theta(x;0)=1-x$.

As all readers may not be so comfortable with infinite products, we give a direct proof of the following fact.

\begin{lemma}\label{qpl}
For $|p|<1$, $(x;p)_\infty$ is an entire function of $x$ with zeroes precisely at $x\in p^{\mathbb Z_{\leq 0}}$.
\end{lemma}

\begin{proof}
We start from the Taylor expansion
$$\log \frac 1{1-x}=\sum_{n=1}^\infty\frac{x^n}n,\qquad |x|<1, $$
which gives
$$1-x=\exp\left(-\sum_{n=1}^\infty\frac{x^n}n\right),\qquad |x|<1. $$
Fixing $x$, pick $N$ so that $|xp^{N+1}|<1$. We can then write
$$(x;p)_\infty=\prod_{j=0}^N(1-xp^j)\exp\left(-\sum_{j=N+1}^\infty\sum_{n=1}^\infty\frac{(xp^j)^n}n\right). $$
As the double series converges absolutely we may change the order of summation and obtain
$$(x;p)_\infty=\prod_{j=0}^N(1-xp^j)\exp\left(-\sum_{n=1}^\infty\frac{x^np^{(N+1)n}}{n(1-p^n)}\right). $$
The stated properties are then obvious.
\end{proof}

We have the following immediate consequence. 

\begin{corollary}\label{tzc}
The theta function $\theta(x;p)$ is analytic for $x\neq 0$ and has zeroes precisely at $x\in p^{\mathbb Z}$. 
\end{corollary}

Note that the theta function is \emph{not} elliptic. In fact,
$$\frac{\theta(px;p)}{\theta(x;p)}=\frac{(px,1/x;p)_\infty}{(x,p/x;p)_\infty}=\frac{1-1/x}{1-x}=-\frac 1{x}. $$
This relation,  
\begin{equation}\label{tqp}\theta(px;p)=-x^{-1}\theta(x;p),\end{equation}
 is called  \emph{quasi-periodicity} of the theta function.
 More generally,
\begin{equation}\label{tqpi} \theta(p^kx;p)=(-1)^kp^{-\binom k2}x^{-k}\theta(x;p),\qquad k\in\mathbb Z.\end{equation}
Another useful identity is
$$\theta(1/x;p)=-x^{-1}\theta(x;p).$$

\begin{problem}
Prove that $\prod_{n=2}^N(1-1/n)\rightarrow 0$ as $N\rightarrow\infty$. (This shows that one has to be a little bit careful when showing that $\theta(x;p)\neq 0$ for $x\notin p^{\mathbb Z}$; it does not just follow from the fact that the factors are non-zero and tend to $1$.)
\end{problem}
\begin{problem}
Prove \eqref{tqpi}.
\end{problem}
\begin{problem}\label{tde}
Show that $\theta(x^2;p^2)=\theta(x,-x;p)$ and that $\theta(x;p)=\theta(x,px;p^2)$. Deduce the duplication formula
\begin{equation}\label{tdf}\theta(x^2;p)=\theta(x,-x,\sqrt px,-\sqrt px;p).\end{equation}
\end{problem}
\begin{problem} 
Show that 
\begin{equation}\label{tte}\theta(-1,\sqrt p,-\sqrt p;p)=2,\end{equation}
 first using \eqref{tdf} and then by direct manipulation of infinite products.
\end{problem}

\begin{problem}\label{jte}
Show that $\theta(x;p)$ has the Laurent expansion\footnote{This is known as Jacobi's triple product identity. A neat way to compute the prefactor is to compare the cases  $x=\ti \sqrt p$ and $x=\sqrt p$, see \cite[\S 10.4]{aar}.}
$$\theta(x;p)=\frac{1}{(p;p)_\infty}\sum_{n=-\infty}^\infty(-1)^np^{\binom n 2} x^n. $$
\end{problem}

\section{Factorization of elliptic functions}

We will now show that elliptic functions can be factored in terms of theta functions.
We first recall the following elementary fact, which is easily proved by expanding $f$ as a Laurent series.

\begin{lemma}\label{ll}
If $f$ is analytic on $\mathbb C^\ast$ and $f(px)=Cf(x)$ for some  $C\in\mathbb C$ and $|p|<1$, then $f(x)=Dx^N$ for some $D\in\mathbb C$ and $N\in\mathbb Z$.
\end{lemma}

The following result will be useful.

\begin{lemma}\label{zpl}
Let $f$ be multiplicatively elliptic with period $p$. Then, $f$ has as many   poles as zeroes, counted with multiplicity, in  each period annulus $A=\left\{x;\,pr\leq |x|<r\right\}$. 
\end{lemma}

\begin{proof}
We can assume that there are no zeroes or poles at $\partial A$; otherwise we just vary $r$ slightly. By the argument principle, if $N$ is the number of zeroes and $P$ the number of poles inside  $A$, then
$$N-P=\int_{\partial A}\frac{f'(x)}{f(x)}\frac{dx}{2\pi\ti }. $$
Here, the inner boundary circle should be oriented clockwise and the outer circle counter-clockwise.
To compare the two components we change $x$ to $px$, as well as the orientation, at the inner boundary. 
Since $f(px)=f(x)$ gives $pf'(px)=f'(x)$, it follows that
$$N-P=\int_{|x|=r} \left(\frac{f'(x)}{f(x)}-\frac{pf'(px)}{f(px)}\right)\frac{dx}{2\pi\ti }=0.$$
\end{proof}

We can now obtain the following fundamental result.

\begin{theorem}\label{eftt}
Any multiplicatively elliptic function $f$ with period $p$  can be factored as
\begin{subequations}\label{eft}
\begin{equation}f(x)=C\frac{\theta(x/a_1,\dots,x/a_n;p)}{\theta(x/b_1,\dots,x/b_n;p)}, \end{equation}
where $C\in \mathbb C$  and $a_j,b_j\in\mathbb C^\ast$ are subject to the condition
\begin{equation}a_1\dotsm a_n=b_1\dotsm b_n. \end{equation}
\end{subequations}
\end{theorem}

\begin{proof}
Pick a period annulus $A$ such that $f$ has no zeroes or poles at the boundary.  
By Lemma \ref{zpl}, there are as many zeroes as poles inside $A$ (counted with multiplicity); denote them $a_1,\dots,a_n$ and $b_1,\dots,b_n$, respectively. 
All zeroes and poles  are then of the form $p^{\mathbb Z}a_j$ and   $p^{\mathbb Z}b_j$ .
Thus, by Corollary \ref{tzc},
$$g(x)=f(x)\frac{\theta(x/b_1,\dots,x/b_n;p)}{\theta(x/a_1,\dots,x/a_n;p)} $$
is analytic for $x\neq 0$. Using \eqref{tqp} it follows that $g(px)=Dg(x)$, where $D=b_1\dotsm b_n/a_1\dotsm a_n$. 
By Lemma \ref{ll}, $g(x)=C x^N$ and $D=p^N$ for some  $C\in\mathbb C$ and $N\in\mathbb Z$. We have now proved that
$$f(x)=Cx^N \frac{\theta(x/a_1,\dots,x/a_n ;p)}{\theta(x/b_1,\dots,x/b_n;p)}, \qquad p^Na_1\dotsm a_n=b_1\dotsm b_n.$$
Replacing $a_1$ by $p^{-N}a_1$ and using \eqref{tqpi} we arrive at \eqref{eft}.
\end{proof}

Note   that the  limit $p\rightarrow 0$ of \eqref{eft} does not give all rational function but only those of the form
$$f(x)=C\frac{(x-a_1)\dotsm(x-a_n)}{(x-b_1)\dotsm(x-b_n)},\qquad a_1\dotsm a_n=b_1\dotsm b_n, $$
that is, rational functions such that $f(0)=f(\infty)$. From this perspective, it is natural to consider any function of the form
\begin{equation}\label{etk}C\frac{\theta(x/a_1,\dots,x/a_m;p)}{\theta(x/b_1,\dots,x/b_n;p)}\end{equation}
as ``kind of elliptic''. Indeed, such functions are sometimes called \emph{elliptic functions of the third kind}. The special case $m=n$ is then called \emph{elliptic functions of the second kind} and the true elliptic functions, satisfying in addition 
$a_1\dotsm a_n=b_1\dotsm b_n$, are \emph{elliptic functions of the first kind}. Moreover, the special case $n=0$ (corresponding to polynomials in the trigonometric limit) is referred to as \emph{higher order theta functions}, or simply \emph{theta functions}. 

We state an extension of Theorem \ref{eftt} to elliptic functions of the third kind, but leave the proof to the reader.

\begin{theorem}\label{etkt}
Let $f$ be a meromorphic function on $\mathbb C^\ast$ satisfying the equation $f(px)=t x^{-k}f(x)$, where $k\in\mathbb Z$, $t\in\mathbb C^\ast$ and $0<|p|<1$. Then, $f(x)$ can be factored as in \eqref{etk}, where $m=n+k$ and $(-1)^ka_1\dotsm a_m/b_1\dotsm b_n=t$.
\end{theorem}

The special case of higher order theta functions is as follows.

\begin{corollary}\label{tfc}
Let $f$ be an analytic function on $\mathbb C^\ast$ satisfying the equation $f(px)=t x^{-k}f(x)$, where $k\in\mathbb Z$, $t\in\mathbb C^\ast$ and $0<|p|<1$. Then, $k\geq 0$ and
$$f(x)=C\theta(x/a_1,\dots,x/a_k;p), \qquad (-1)^ka_1\dotsm a_k=t.$$
\end{corollary}

\begin{problem}
Show that an elliptic function assumes each value (including $\infty$) an equal number of times in each period annulus $|pr|\leq |z|<r$.
\end{problem}

\begin{problem}
Prove  Theorem \ref{etkt}. 
\end{problem}

\section{The three-term identity}\label{ttrs}

From the viewpoint of elliptic hypergeometric series, the most fundamental result on elliptic functions is a certain three-term
relation for theta functions due to Weierstrass.\footnote{It has also been attributed to Riemann, but that seems incorrect \cite{k}.}
To motivate this relation, consider the space $V$ of analytic functions on $\mathbb C^\ast$ satisfying $f(px)=f(x)/px^2$. 
By Corollary \ref{tfc}, it consists of functions of the form $C\theta(xa,px/a;p)=C\theta(ax^\pm;p)$ (recall the notation \eqref{tpm}). 
As it is described by two parameters, we expect that $\dim V=2$, so $V$ should have a basis of the form $\theta(bx^\pm;p)$
and $\theta(cx^\pm;p)$. Thus, we should be able to write
\begin{equation}\label{we}\theta(ax^\pm;p)=B\theta(bx^\pm;p)+C\theta(cx^\pm;p). \end{equation}
If we put $x=c$ we get $B=\theta(ac^\pm;p)/\theta(bc^\pm;p)$, provided that the denominator is non-zero.
Similarly, $C=\theta(ab^\pm;p)/\theta(cb^\pm;p)$. Clearing the denominator, we are led to the identity
\begin{equation}\label{ttr}\theta(ax^\pm,bc^\pm;p)=\theta(bx^\pm,ac^\pm;p)+\frac ac\,\theta(cx^\pm,ba^\pm;p). \end{equation}

Although it is not hard to make the above argument rigorous, let us give an independent proof of \eqref{ttr} from scratch.
Let $f(x)$ denote the difference of the left-hand and right-hand side in \eqref{ttr}. We may assume that all parameters are generic.
It is clear that $f(c)=f(c^{-1})=0$. Since $f(px)=f(x)/px^2$, $f$ vanishes at $c^{\pm}p^{\mathbb Z}$. It follows that
$g(x)=f(x)/\theta(cx,c/x;p)$ is analytic for $x\neq 0$. Moreover, $g(px)=g(x)$. By Liouville's theorem\footnote{In complex analysis you have probably learned that Liouville's theorem says that 
entire bounded functions are constant. What Liouville in fact proved was the weaker statement that entire elliptic functions are constant (this follows from Theorem \ref{eftt}). The generalization to bounded functions is due to Cauchy.}, a non-constant elliptic function must have poles, so $g$ is a constant. But since we also have $f(b)=0$, that constant must be zero. Hence, $f$ is identically zero.

\begin{problem}
Deduce from \eqref{ttr} that  the functions  $\theta(bx^\pm;p)$
and $\theta(cx^\pm;p)$ form a basis for the space $V$ if and only if $bc,\ b/c\notin p^{\mathbb Z}$; in  particular, $\dim V=2$.
\end{problem}

\begin{problem}
Prove that the ``elliptic number'' $[z]=e^{-\ti\pi z}\theta(e^{2\pi\ti z};e^{2\pi\ti\tau})$ satisfies
\begin{equation}\label{na}[z+a][z-a][b+c][b-c]=[z+b][z-b][a+c][a-c]+[z+c][z-c][b+a][b-a]. \end{equation}
Deduce as limit cases that the same identity holds for the ``trigonometric number'' $[z]=\sin(z)$ and the ``rational number'' $[z]=z$.\footnote{Any entire function satisfying \eqref{na} is of one of these three forms, up to the transformations $[z]\mapsto ae^{bz^2}[cz]$, see \cite[Ex.\ 20.38]{ww}.}
\end{problem}

\begin{problem}\label{tti}
Show that the trigonometric and rational numbers in the previous exercise satisfy
\begin{equation}\label{ta}[b+c][b-c]=[a+c][a-c]+[b+a][b-a],\end{equation}
but that this is \emph{not} true for the elliptic numbers (one way to see this is to shift one of the variables by $\tau$).
\end{problem}

\section{Even elliptic functions}\label{ees}

It will be useful to understand the structure of even additively elliptic functions.
Note that if $f(z)=g(e^{2\ti\pi z})$ is such a function, then the corresponding multiplicatively elliptic function $g$ satisfies 
 $g(1/x)=g(x)$. By slight abuse of terminology, we will use the word \emph{even} also for the latter type of symmetry.

\begin{lemma}\label{eml}
Let  $g$ be an even multiplicatively elliptic function, that is, $g$ is meromorphic  on $\mathbb C^\ast$  and satisfies
\begin{equation}\label{gpi}g(px)=g(1/x)=g(x).\end{equation}
Then, if $a^2\in p^{\mathbb Z}$, the multiplicity of $a$ as a zero or pole  of $g$ is even.
\end{lemma}

\begin{proof}
It follows from \eqref{gpi} that $g(ax)=g(1/ax)=g(a/x)$. Suppose the Laurent expansion of $g$ near $a$ starts as $C(x-a)^j$. Then,
$(ax-a)^j\sim (a/x-a)^j$ as $x\rightarrow 1$, which is only possible for $j$ even.
\end{proof}

We can now give a counterpart of Theorem \ref{eftt} for even elliptic functions.

 \begin{proposition}\label{efp}
Any even multiplicatively elliptic function $g$
 can be factored as
\begin{equation}\label{etf}g(x)=C\frac{\theta(c_1x^\pm,\dots,c_m x^\pm;p)}{\theta(d_1x^\pm,\dots,d_m x^\pm;p)},\end{equation}
where $C\in\mathbb C$ and $c_j,d_j\in\mathbb C^\ast$. 
\end{proposition}

 \begin{proof}
 We first factor $g$ as in \eqref{eft}. Since 
  $g(1/a_1)=g(a_1)=0$, we must have $a_1a_j\in p^{\mathbb Z}$ for some $j$. 
  Consider first the case $j\neq 1$. Then, 
  $\theta(x/a_1,x/a_j;p)\sim \theta(a_1 x^\pm;p)$, where $\sim$ means equality up to a factor of the form
  $C x^k$. If $j=1$,  it follows from Lemma \ref{eml} that 
 the multiplicity of $a_1$ as a zero of $g$ is  even. 
  If the multiplicity is $2l$, this leads to a factor $\sim \theta(a_1x^\pm;p)^l$. The same argument applies to the poles of $g$, so in conclusion we
  find that
  $$g(x)=C x^k\frac{\theta(c_1x^\pm,\dots,c_m x^\pm;p)}{\theta(d_1x^\pm,\dots,d_m x^\pm;p)}$$
  for some integer $k$.  
It is easy to check that \eqref{gpi} holds if and only if $k=0$.
 \end{proof}

This factorization has the following important consequence.

\begin{proposition}\label{uec}
Let
\begin{equation}\label{xe}X(x)=\frac{\theta(ax^\pm;p)}{\theta(bx^\pm;p)}, \end{equation}
where    $ab,a/b\notin p^\mathbb Z$. Then, $X$ generates the field of even multiplicatively elliptic functions, that is,
any such function $g$ is of the form $g(x)=p(X(x))$, with $p$ a rational function.
\end{proposition}

\begin{proof}
Starting from \eqref{etf}, we order the parameters so that $c_j\in bp^\mathbb Z$ if and only if $1\leq j\leq k$ and 
$d_j\in b p^\mathbb Z$ if and only if $1\leq j\leq l$. Using \eqref{tqpi}, we can assume that the remaining parameters $c_j$ and $d_j$ equal $b$, so that
$$g(x)=C\theta(bx^\pm;p)^{l-k}\frac{\theta(c_1x^\pm,\dots,c_kx^\pm;p)}{\theta(d_1x^\pm,\dots,d_lx^\pm;p)}. $$
On the other hand, \eqref{ttr} gives
\begin{equation}\label{xd}X(x)-X(c)=\frac{a\theta(ba^\pm,cx^\pm;p)}{c\theta(bc^\pm,bx^\pm;p)}. \end{equation}
It follows that
$$g(x)=D\frac{(X(x)-X(c_1))\dotsm(X(x)-X(c_k))}{(X(x)-X(d_1))\dotsm(X(x)-X(d_l))},
 $$
where $D$ is a non-zero constant. 
\end{proof}

That the field of even elliptic functions is generated by a single element can also be understood
 geometrically. Such functions live on the quotient $S$ of the torus by the relation $x=x^{-1}$
(or, additively, $z=-z$). It can be shown that $S$ is a sphere, so there must exist a holomorphic bijection (known as a uniformizing map)
$X$ from $S$ to the Riemann sphere $\mathbb C\cup\{\infty\}$. The meromorphic functions 
 on the Riemann sphere are simply the rational functions. Consequently, the meromorphic functions on $S$ are precisely the rational functions in
 $X(x)$.

\begin{problem}
By drawing pictures, convince yourself that the quotient $S$ discussed in the text is a topological sphere.
\end{problem}

\begin{problem}\label{ete}
Generalize Proposition \ref{efp} to elliptic functions of the third kind.
\end{problem}

\begin{problem}\label{xde}
Deduce from \eqref{xd} that
$$X'(x)= \frac{a(p;p)_\infty^2\theta(ba^\pm,x^2;p)}{x^2\theta(bx^{\pm};p)^2}.$$
\end{problem}

\begin{problem}\label{pxe}
Suppose that $f$ is analytic on $\mathbb C^\ast $ and satisfies $f(px)=f(x)/x^np^{2n}$ and $f(1/x)=f(x)$, where $n\in\mathbb Z_{\geq 0}$. Show that
$$f(x)=\theta(bx^\pm;p)^n p(X(x)), $$
where $p$ is a unique polynomial of degree at most $n$ and $X$ is as in \eqref{xe}. In particular, the space of such functions has dimension $n+1$.
\end{problem}

\section{Interpolation and partial fractions}\label{pfs}

Lagrange interpolation expresses a polynomial 
 of degree $n-1$ in terms of its values at $n$ distinct points. 
If the points are $(y_j)_{j=1}^{n}$, one introduces the polynomials $(p_j)_{j=1}^n$ by
$$p_j(x)=\prod_{k=1,\, k\neq j}^n (x-y_k). $$
Note that $p_j(y_k)\neq 0$ if and only if $k=j$. Thus, if 
\begin{equation}\label{pep}p(x)=\sum_{j=1}^n c_j p_j(x) \end{equation}
then $c_k=p(y_k)/p_k(y_k)$. In particular, choosing $p$ as the zero polynomial it follows that  $(p_j)_{j=1}^n$ are linearly independent.
Counting dimensions, they form a basis for the polynomials of degree at most $n-1$, so any such polynomial $p$ can be expanded as in \eqref{pep}.
This yields the interpolation formula
$$p(x)=\sum_{j=1}^n p(y_j)\prod_{k=1,\, k\neq j}^n\frac{x-y_k}{ y_j-y_k}. $$ 
If we let $p(x)=\prod_{k=1}^{n-1}(x-z_k)$  and divide by $\prod_{k=1}^n(x-y_k)$ we get the partial fraction expansion
\begin{equation}\label{rpf}\frac{\prod_{k=1}^{n-1} (x-z_k)}{\prod_{k=1}^n(x-y_k)}=\sum_{j=1}^n\frac{\prod_{k=1}^{n-1}(y_j-z_k)}{\prod_{k=1,\,k\neq j}^n(y_j-y_k)}\cdot\frac 1{x-y_j}, \end{equation}
which is useful for integrating rational functions.

Lagrange interpolation also works for theta functions; in fact, we have already seen an example in \S \ref{ttrs}. It may seem natural to replace the polynomials $p_j$ with the theta functions
 $f_j(x)=\prod_{k\neq j}\theta(x/y_k;p)$. However, these functions satisfy different quasi-periodicity relations
and thus don't span a very natural space. Instead, we take $f_j(x)=\theta(tx/y_j;p)\prod_{k\neq j}\theta(x/y_k;p)$. If we first let $p=0$ and then $t=0$ we recover the polynomials $p_j$. We then have the following fact.

\begin{proposition}\label{elp}
Let $t,y_1,\dots,y_n\in\mathbb C^\ast$ be such that neither $t$ nor $y_j/y_k$ for $j\neq k$ is in $p^{\mathbb Z}$. 
Let $V$ be the space of functions that are analytic for $x\neq 0$ and satisfy 
$f(px)=(-1)^n y_1\dotsm y_n t^{-1}x^{-n}f(x)$.  
Then, any $f\in V$ is uniquely determined by the values $f(y_1),\dots,f(y_n)$ and given  by
\begin{equation}\label{eli}f(x)=\sum_{j=1}^nf(y_j)\frac{\theta(tx/y_j;p)}{\theta(t;p)}\prod_{k=1,\,k\neq j}^n\frac{\theta(x/y_k;p)}{\theta(y_j/y_k;p)}.  \end{equation}
\end{proposition}

\begin{proof}
With $f_j$ as defined above,
it is easy to see that $f_j\in V$. By the conditions on the parameters,  $f_j(y_k)\neq 0$ if and only if $k=j$. Take now $f\in V$ and consider
$$g(x)=f(x)-\sum_{j=1}^n\frac{f(y_j)}{f_j(y_j)}\,f_j(x). $$
Then, $g\in V$ and $g$ vanishes at $x=y_1,\dots,y_n$. By quasi-periodicity, it vanishes at $p^{\mathbb Z}y_j$, so
$$h(x)=\frac{g(x)}{\theta(x/y_1,\dots,x/y_n;p)} $$  is analytic on $\mathbb C^\ast$. 
Moreover, $h(px)=h(x)/t$.  By Lemma \ref{ll}, we can write $h(x)=Cx^N$ with $N\in\mathbb Z$. Since $t\notin p^{\mathbb Z}$, we must have $C=0$ and consequently
$$f(x)=\sum_{j=1}^n\frac{f(y_j)}{f_j(y_j)}\,f_j(x). $$
Writing this out explicitly gives \eqref{eli}.
\end{proof}

Note that it follows that $(f_j)_{j=1}^n$ form a basis for $V$ and, in particular, that $\dim V=n$. This is expected since, by 
Corollary \ref{tfc}, any $f\in V$ is of the form $f(x)=C\prod_{k=1}^n\theta(x/z_k)$, where $tz_1\dotsm z_n=y_1\dotsm y_n$, and is thus described by $n$ free parameters. Inserting this factorization in Proposition \ref{elp}
and dividing by $\prod_{k=1}^n\theta(x/y_k)$, we obtain the elliptic partial fraction expansion
\begin{equation}\label{epf}\prod_{k=1}^n\frac{\theta(x/z_k;p)}{\theta(x/y_k;p)}=\sum_{j=1}^n\frac{\prod_{k=1}^n\theta(y_j/z_k;p)}{\prod_{k=1,\,k\neq j}^n\theta(y_j/y_k;p)}\cdot\frac{\theta(xy_1\dotsm y_n/y_jz_1\dotsm z_n;p)}{\theta(y_1\dotsm y_n/z_1\dotsm z_n,x/y_j;p)}. \end{equation}
If we let $x=z_n$, cancel all factors involving $z_n$ and then introduce a new variable  $z_n=y_1\dotsm y_n/z_1\dotsm z_{n-1}$,  we obtain the elegant identity\footnote{The earliest reference I have found is \cite[p.\ 46]{t}.}
\begin{equation}\label{epv}\sum_{j=1}^n\frac{\prod_{k=1}^{n}\theta(y_j/z_k;p)}{\prod_{k=1,\,k\neq j}^n\theta(y_j/y_k;p)}=0,\qquad y_1\dotsm y_n=z_1\dotsm z_n, \end{equation}
which is in fact equivalent to \eqref{epf}.

We will also need another elliptic partial fraction expansion, connected with even elliptic functions. Namely,
\begin{equation}\label{dpf}\frac{\prod_{k=1}^{n-1}\theta(xz_k^\pm;p)}{\prod_{k=1}^n\theta(x y_k^\pm;p)}=\sum_{j=1}^n\frac {\prod_{k=1}^{n-1}\theta(y_jz_k^{\pm};p)}{\theta(xy_j^\pm;p)\prod_{k=1,\,k\neq j}^n\theta(y_jy_k^\pm;p)}. \end{equation}
The proof is left to the reader as Exercise \ref{dpe}.
The special case $x=z_1$ is
\begin{equation}\label{vdf} \sum_{j=1}^n\frac {y_j\prod_{k=2}^{n-1}\theta(y_jz_k^{\pm};p)}{\prod_{k=1,\,k\neq j}^n\theta(y_jy_k^\pm;p)}=0, \qquad n\geq 2.\end{equation}
Again, this is equivalent to the general case.

\begin{problem}
Show that \eqref{epf} is equivalent to \eqref{epv} (with $n$ replaced by $n+1$).
\end{problem}

\begin{problem}
Show that the case $n=2$ of \eqref{epf} and \eqref{dpf} are both equivalent to Weierstrass's identity \eqref{ttr}.
\end{problem}

\begin{problem}\label{dpe} Give two proofs of \eqref{dpf}. First, imitate the proof of \eqref{epf}, using the basis $f_j(x)=\prod_{k\neq j}\theta(y_k x^\pm;p)$ 
for an appropriate space of theta functions. Second, substitute $x=X(x)$, $y_k=X(y_k)$, $z_k=X(z_k)$ in \eqref{rpf}, where $X$ is as in \eqref{xe}.
\end{problem}

\begin{problem}
Show in two ways that, for $a_1\dotsm a_nb_1\dotsm b_{n+2}=1$,
\begin{multline}\label{sbt}
x^{-n-1}\theta(a_1x,\dots,a_nx,b_1x,\dots,b_{n+2}x;p)\\
-x^{n+1}\theta(a_1x^{-1},\dots,a_nx^{-1},b_1x^{-1},\dots,b_{n+2}x^{-1};p)\\
=\frac{(-1)^nx\theta(x^{-2};p)}{a_1\dotsm a_n}\sum_{k=1}^n\prod_{j=1}^{n+2}\theta(a_kb_j;p)\prod_{j=1,\,j\neq k}^n\frac{\theta(a_jx^\pm;p)}{\theta(a_k/a_j;p)}.
\end{multline}
First, prove that \eqref{sbt} is equivalent to \eqref{epf}, then prove it directly by viewing it as an interpolation formula for functions in $x$.
\end{problem}

\begin{problem}
Let $V$ be the vector space of functions satisfying the conditions of Corollary \ref{tfc}. By expanding the elements of $V$ as Laurent series and using Exercise~\ref{jte}, show that the functions 
$f_j(x)=x^j\theta(-p^jx^n/t;p^n)$, $j=1,\dots,n$,
form a basis for $V$. (This gives an independent proof that $\dim V=n$.)
\end{problem}

\begin{problem}
Use \eqref{epf} to prove Frobenius's determinant evaluation\footnote{See \cite{kn} for  applications to multivariable elliptic hypergeometric series.}
\begin{multline*}\det_{1\leq i,j\leq n}\left(\frac{\theta(tx_iy_j;p)}{\theta(x_iy_j;p)}\right)\\
=\frac{\theta(t;p)^{n-1}\theta(tx_1\dotsm x_ny_1\dotsm y_n;p)\prod_{1\leq i<j\leq n}x_jy_j\theta(x_i/x_j,y_i/y_j;p)}{\prod_{i,j=1}^n\theta(x_iy_j;p)}. \end{multline*}
\end{problem}

\section{Modularity and elliptic curves}
We have now presented a minimum of material on elliptic functions needed for the remainder of these notes. We proceed to 
discuss some topics that are more peripheral to our main purpose, but so central in other contexts that we cannot
ignore them completely. We start with the following important fact. 

\begin{theorem}\label{msl}
If $\tau$ and $\tau'$ are in the upper half-plane, the corresponding complex tori $E_\tau=\mathbb C/(\mathbb Z+\tau\mathbb Z)$
and $E_{\tau'}$ are equivalent as Riemann surfaces if and only if 
 $\tau'=(a\tau+b)/(c\tau+d)$ for some integers $a,\,b,\,c,\,d$ with $ad-bc=1$. If that is the case then $\phi(z)=z/(c\tau+d)$ gives an equivalence
$E_\tau\rightarrow E_{\tau'}$.
\end{theorem}

In this context, $\tau$ and $\tau'$ are called  \emph{moduli}\footnote{The word modulus is also used for other quantities parametrized by $\tau$, cf.\ Exercise \ref{scde}.} and the map $\tau\mapsto \tau'$ a \emph{modular transformation}. 
The rule for composing such maps is the same as for multiplying the matrices $\left(\begin{smallmatrix}a&b\\c&d\end{smallmatrix}\right)$,
which form the \emph{modular group} $\mathrm{SL}(2,\mathbb Z)$. (More precisely, the group of modular transformations is isomorphic to $\mathrm{SL}(2,\mathbb Z)/\{\pm 1\}$.)
  The lemma states that
we can parametrize the space of complex tori by identifying any two  moduli related by a modular transformation; the
corresponding quotient of the upper half-plane is called  the moduli space. The term moduli space is nowadays used more generally for any space parametrizing geometric objects, but the  origin of the term comes from the special case considered here.

\begin{proof}[Proof of \emph{Theorem \ref{msl}}]
Let $\phi$ be an invertible analytic map from $E_\tau$ to $E_{\tau'}$. Equivalently, $\phi$ is an invertible entire function such that
$$\phi(z+\mathbb Z+\tau\mathbb Z)= \phi(z)+\mathbb Z+\tau'\mathbb Z. $$
Since any invertible entire function has the form $\phi(z)=Cz+D$, $C\neq 0$, we get
\begin{equation}\label{mr}C(\mathbb Z+\tau\mathbb Z)= \mathbb Z+\tau'\mathbb Z. \end{equation}
Choosing $1$ in the right-hand side gives $C=1/(c\tau+d)$ for some integers $c$, $d$. If we then choose $\tau'$ in the right-hand side we find that
$\tau'=(a\tau+b)/(c\tau+d)$ for some integers $a$, $b$. Interchanging the roles of $\tau$ and $\tau'$ gives $\tau=(A\tau+B)/(C\tau +D)$ for integers $A$, $B$, $C$, $D$, where necessarily
$$\left(\begin{matrix}a&b\\c&d\end{matrix}\right)\left(\begin{matrix}A&B\\C&D\end{matrix}\right)=\left(\begin{matrix}1&0\\0&1\end{matrix}\right). $$ 
Taking determinants gives $(ad-bc)(AD-BC)=1$ and hence
$ad-bc=\pm 1$. We may compute
 $$\operatorname{Im}(\tau')=\frac{ad-bc}{|c\tau+d|^2}\operatorname{Im}(\tau)$$
and conclude that $ad-bc=1$.  The final statement is easy to check.
\end{proof}

Since meromorphic functions on $E_\tau$  can be factored in terms of theta functions, one would expect that
$\theta(e^{2\ti\pi z};e^{2\ti\pi\tau})$ transforms nicely under the modular group.
Indeed, if $\left(\begin{smallmatrix}a&b\\c&d\end{smallmatrix}\right)\in\mathrm{SL}(2,\mathbb Z)$, then
\begin{equation}\label{tmt}e^{-\ti\pi z/(c\tau+d)}\theta(e^{2\ti\pi z/(c\tau+d)};e^{2\ti\pi(a\tau+b)/(c\tau+d)})=Ce^{\ti\pi cz^2/(c\tau+d)-\ti\pi z}\theta(e^{2\ti\pi z};e^{2\ti\pi\tau}), \end{equation}
where $C=C(a,b,c,d;\tau)$ is independent of $z$. To see this, simply observe that if $f(z)$ is the quotient of the left-hand and right-hand sides, then
$f$ is an entire function, as both sides vanish precisely at $\mathbb Z+\tau\mathbb Z$. Moreover, $f$ is periodic with periods $1$ and $\tau$ so, by
Liouville's theorem, $f$ is constant.

Explicit expressions for the constant in \eqref{tmt} exist but are somewhat complicated \cite[\S 80]{rad}.
Usually, one is content with giving it for the transformations $\tau\mapsto \tau+1$ and $\tau\mapsto -1/\tau$, which are known to generate the modular group.
In the first case, $C=1$ trivially; the second case is treated in Exercise \ref{mte}.

In Section \ref{ees} we showed that the field of even elliptic functions is generated by a single element. 
A slight extension of the  argument gives an analogous statement for general elliptic functions.

\begin{theorem}\label{ecp}
The field of all elliptic functions, with periods $1$ and $\tau$, is isomorphic to the quotient field $\mathbb C(X,Y)/(Y^2-X(X-1)(X-\lambda))$, where
 $\lambda=\lambda(\tau)$ is a certain function (the lambda invariant).\footnote{\label{lp} There are six possible choices for $\lambda$. The standard one is derived in Exercise~\ref{ele}:
$$\lambda(\tau)=16\sqrt p\frac{(-p;p)_\infty^8}{(-\sqrt p;p)_\infty^8}=16q-128 q^2+704q^3-\dotsm,\qquad q=\sqrt p=e^{\ti\pi\tau}. $$} 
\end{theorem}

\begin{proof}
We will only sketch the proof. Consider first an 
odd multiplicatively elliptic function $g$, that is, $g(px)=-g(1/x)=g(x)$. Let $h(x)=x g(x)/\theta(x^2;p)$. Then, $h(1/x)=h(x)$ and $h(px)=p^2x^4h(x)$.
As in the proof of Proposition~\ref{efp} (cf.~Exercise \ref{ete}), it follows that 
$$g(x)=C\frac{\theta(x^2,c_1x^\pm,\dots,c_{m-2}x^\pm;p)}{x\,\theta(d_1x^\pm,\dots,d_m x^\pm;p)}. $$ 
If we take $X(x)$ as in  Proposition \ref{uec} and
$Y(x)={\theta(x^2;p)}/{x\theta(bx^\pm;p)^2}$,
it follows that
$g(x)=Y(x)p(X(x))$,
where $p$ is a rational function.
As any function is the sum of an even and an odd function, and the even elliptic functions are rational in $X$, it follows that the field of all elliptic functions is generated by $X$ and $Y$.\footnote{By Exercise \ref{xde}, we can always choose $Y(x)=xX'(x)$.} 

Let us now consider the possible relations between $X$ and $Y$. Since $Y^2$ is even, it is a rational function of $X$. 
If we choose $a$ and $b$ in the definition of $X$ and $Y$ as two of the numbers $\{1,-1,\sqrt p,-\sqrt p\}$,
then it follows from the proof of Proposition~\ref{uec} that
$$Y^2=CX(X-A)(X-B), $$
where $A$ and $B$ are the values of $X$ at the remaining two numbers and $C\neq 0$. By rescaling $X$ and $Y$, this can be reduced to
$Y^2=X(X-1)(X-\lambda)$. 
It remains to prove that there are no further algebraic relations between $X$ and $Y$. To this end, suppose that
$P(X(x),Y(x))\equiv 0$ for some rational $P$. We can  write
$$P(X,Y)=Q(X,Y^2)+YR(X,Y^2) $$
with $Q$ and $R$ rational. Let  $E(X)=X(X-1)(X-\lambda)$. Taking the even and odd part of
$P(X(x),Y(x))$, it follows that $Q(X(x),E(X(x)))=R(X(x),E(X(x)))=0$. But, by the proof of Proposition~\ref{uec},  the 
zero elliptic function can only be obtained from the zero rational function, so $Q(X,E(X))=R(X,E(X))=0$. By the factor theorem,
$Q(X,Y)$ and $R(X,Y)$ are divisible by $Y-E(X)$, hence $P(X,Y)$ is divisible by $Y^2-E(X)$. 
\end{proof}

Theorem \ref{ecp} shows that a complex torus is an algebraic variety, known as an \emph{elliptic curve}.

\begin{problem}\label{mte}
Compute the constant $C(\tau)$ in \eqref{tmt} for the transformation $\tau\mapsto -1/\tau$. Indeed, show that 
$$ e^{-\ti\pi z/\tau}\theta(e^{2\ti\pi z/\tau};e^{-2\ti\pi/\tau})=-\ti e^{\ti\pi(\tau+\tau^{-1})/6+\ti\pi z^2/\tau-\ti\pi z}\theta(e^{2\ti\pi z};e^{2\ti\pi\tau}).$$
One way is to first specialize $z$ to the three values  $1/2$, $\tau/2$ and $(\tau+1)/2$. Then apply  \eqref{tte} to get an expression for $C(\tau)^3$ and finally let $\tau=\ti$ to find the correct branch of the 
cubic root.\footnote{This is the famous imaginary transformation of Jacobi. It is often proved in more complicated ways, 
by authors insisting on using the series representation for $\theta$ given in Exercise~\ref{jte}.
The elementary proof sketched here is taken from the classic textbook \cite{ww}.} 
\end{problem}

\begin{problem}
By letting $z\rightarrow 0$ in the previous exercise, prove  that \emph{Dedekind's eta function} $\eta(\tau)=p^{1/24}(p;p)_\infty$, $p=e^{2\ti\pi\tau}$, satisfies
$\eta(-1/\tau)=\sqrt{-\ti\tau}\eta(\tau)$. Using that the modular group is generated by $\tau\mapsto \tau+1$ and $\tau\mapsto -1/\tau$, conclude that the 
\emph{modular discriminant} $\Delta(\tau)=p(p;p)_\infty^{24}$ satisfies\footnote{
  More generally, a \emph{cusp form} is an analytic function on the upper half-plane that vanishes at $p=0$ ($\tau=\ti\infty$) and satisfies
  \eqref{cf} with $12$ replaced by an arbitrary positive integer (the weight). The modular discriminant is the simplest cusp form in the sense that
  any cusp form of weight $12$ is proportional to $\Delta$ and any cusp form of smaller weight vanishes identically \cite{ap}. }
\end{problem}
\begin{equation}\label{cf}\Delta\left(\frac{a\tau+b}{c\tau+d}\right)=(c\tau+d)^{12}\Delta(\tau),\qquad \left(\begin{matrix}a&b\\c&d\end{matrix}\right)\in\mathrm{SL}(2,\mathbb Z).\end{equation}

\begin{problem}
Let $\psi(x)=\sum_{n=1}^\infty e^{-n^2\pi x}$. 
Combining the previous two problems with Exercise \ref{jte}, show that $\psi(1/x)=\sqrt x\psi(x)+(1-\sqrt x)/2$.
Next, show that
$$ \pi^{-s/2}\Gamma(s/2)\zeta(s)=\int_0^\infty \psi(x)x^{s/2-1}\,dx,\qquad \operatorname{Re}(s)>1,$$
where
$$\Gamma(s)=\int_0^\infty e^{-x} x^{s-1}\,dx,\qquad \zeta(s)=\sum_{n=1}^\infty \frac 1{n^s}.$$
 Deduce that
$$ \pi^{-s/2}\Gamma(s/2)\zeta(s)=-\frac 1s-\frac 1{1-s}+\int_1^\infty\psi(x) \left(x^{s/2}+x^{(1-s)/2}\right)\frac{dx}{x}.$$
It follows that the analytic continuation of the left-hand side is invariant under $s\mapsto 1-s$.\footnote{This is Riemann's original proof of the  functional equation for Riemann's zeta function.}
\end{problem}

\begin{problem}\label{ele}
Let
$$X(x)=\frac{4\sqrt p(-p;p)_\infty^4}{(-\sqrt p;p)_\infty^4}\frac{\theta(-x^\pm;p)}{\theta(-\sqrt p x^\pm;p)}.$$
Using Exercise \ref{xde}, show that 
$$(xX'(x))^2=-CX(X-1)(X-\lambda), $$
where
$$C=(p;p)_\infty^4(-\sqrt p;p)_\infty^8, \qquad
\lambda=16\sqrt p\frac{(-p;p)_\infty^8}{(-\sqrt p;p)_\infty^8}.$$
This gives an explicit expression for the lambda invariant in Theorem \ref{ecp}.
\end{problem}

\begin{problem}
Consider a pendulum released from an angle $\phi_0$. Its subsequent movement is described by the initial value problem
$$\phi''+\frac g l\,\sin\phi=0,\qquad  \phi(0)=\phi_0,\quad \phi'(0)=0.$$
where $\phi$ is the displacement angle, $g$ the gravitational acceleration and $l$ the length of the pendulum. 
Let $Y=(1-\cos\phi)/2=\sin^2(\phi/2)$. Show that
$$(Y')^2=\frac{4g}l\,Y(Y-1)(Y-Y_0),\qquad Y(0)=Y_0=\sin^2(\phi_0/2). $$
Deduce that $Y(t)=X(e^{\ti\mu\, t})$, where $X$ is as in Exercise \ref{ele} with $\mu=\sqrt{4g/lC}$ and $p$ defined implicitly by $\lambda=Y_0$.
In particular, deduce that the period of the pendulum is $2\pi\sqrt{lC/g}$.\footnote{One can give an elegant explanation of the fact that the pendulum is described by elliptic functions as follows. Assume that
$\phi$ is analytic in $t$ and consider $\psi(t)=\pi-\phi(\ti t)$. Then, $\psi$ satisfies the same initial value problem as $\phi$ but with $\phi_0$ replaced by $\pi-\phi_0$. Since we expect that $\phi$ has a real period, so should $\psi$, which leads to an imaginary period for $\phi$.} 
\end{problem}

\begin{problem}
Show that if $\lambda=\lambda(\tau)$ is  as in Theorem \ref{ecp} and 
$$\tilde \lambda(\tau)=\lambda\left(\frac{a\tau+b}{c\tau+d}\right),\qquad \left(\begin{matrix}a&b\\c&d\end{matrix}\right)\in\mathrm{SL}(2,\mathbb Z),$$ then
$$\tilde \lambda\in\left\{\lambda,\ \frac 1{1-\lambda},\ \frac{\lambda-1}\lambda,\ \frac 1\lambda,\ \frac\lambda{\lambda-1},\ 1-\lambda\right\}.$$
\end{problem}

\begin{problem}
Using the previous exercise, show that $j(\tau)=256(1-\mu)^3/\mu^2$, where $\mu=\lambda(1-\lambda)$, is invariant under
$\mathrm{SL}(2,\mathbb Z)$.\footnote{This is the famous fundamental invariant
$$j(\tau)=p^{-1}+744+196884p+\dotsm, \qquad p=e^{2\ti\pi\tau}. $$
One can show that it generates the field of modular functions, that is, the analytic functions on the upper half-plane that are
invariant under $\mathrm{SL}(2,\mathbb Z)$ \cite{ap}.}
\end{problem}

\section{Comparison with classical notation}

   For the benefit of the reader who wants to compare our presentation with the classical approaches of Weierstrass or Jacobi (see e.g.\ \cite{ww})  we provide a few exercises as starting points.

\begin{problem}
In Weierstrass's theory of elliptic functions, the fundamental building block is the $\wp$-function 
$$\wp(z;\tau)=\frac 1{z^2}+\sum_{m,n\in\mathbb Z,\,(m,n)\neq(0,0)}\left(\frac 1{(z+m+n\tau)^2}-\frac{1}{(m+n\tau)^2}\right). $$
Show that $\wp$ is an even additively elliptic function with periods $1$ and $\tau$ and poles precisely at $\mathbb Z+\tau\mathbb Z$.  
Using Proposition~\ref{efp}, deduce that
$$\wp(z;\tau)=C\frac{\theta(ax^\pm;p)}{\theta(x^\pm;p)}, \qquad x=e^{2\ti\pi z},\quad p=e^{2\ti\pi\tau}, $$
for some constants $C$ and $a$ (depending on $\tau$). Deduce from Theorem \ref{ecp} that any additively elliptic function  with periods $1$ and $\tau$
is a rational function of $\wp(z;\tau)$ and its $z$-derivative $\wp'(z;\tau)$.
\end{problem}

\begin{problem}\label{jcte}
Jacobi's theta functions are defined by the Fourier series\footnote{The reader is warned that there are several slightly different versions of these definitions; we follow the conventions of \cite{ww}.}
\begin{align*}
\theta_1(z|\tau)&=2\sum_{n=0}^\infty(-1)^n q^{(n+1/2)^2}\sin((2n+1)z),\\
\theta_2(z|\tau)&=2\sum_{n=0}^\infty q^{(n+1/2)^2}\cos((2n+1)z),\\
\theta_3(z|\tau)&=1+2\sum_{n=1}^\infty q^{n^2}\cos(2nz),\\
\theta_4(z|\tau)&=1+2\sum_{n=1}^\infty(-1)^n q^{n^2}\cos(2nz),\\
\end{align*}
where $q=e^{\ti\pi\tau}$.
Using Exercise \ref{jte}, show that these functions are related to  $\theta(x;p)$ by
\begin{align*}\theta_1(z|\tau)&=\ti q^{1/4}e^{-\ti z}(q^2;q^2)_\infty\theta(e^{2\ti z};q^2),\\
\theta_2(z|\tau)&=q^{1/4}e^{-\ti z}(q^2;q^2)_\infty\theta(-e^{2\ti z};q^2),\\
\theta_3(z|\tau)&=(q^2;q^2)_\infty\theta(-q e^{2\ti z};q^2)_\infty,\\
\theta_4(z|\tau)&=(q^2;q^2)_\infty\theta(qe^{2\ti z};q^2)_\infty.
\end{align*}
\end{problem}

\begin{problem}\label{scde}
Fixing $\tau$ in the upper half-plane and $q=e^{\ti\pi\tau}$, let 
$$k=\frac{\theta_2(0|\tau)^2}{\theta_3(0|\tau)^2}=4q^{1/2}\frac{(-q^2;q^2)_\infty^4}{(-q;q^2)_\infty^4}.$$ 
The parameter $k$ is called the \emph{modulus} in Jacobi's theory of elliptic functions.\footnote{Note that $k$ is related to the lambda invariant of Exercise \ref{ele} by $\lambda=k^2$ (where $p=q^2$).}
This theory is based on the functions
\begin{align*}
\sn(u,k)&=\frac{\theta_3(0|\tau)\theta_1(z|\tau)}{\theta_2(0|\tau)\theta_4(z|\tau)},\\
\cn(u,k)&=\frac{\theta_4(0|\tau)\theta_2(z|\tau)}{\theta_2(0|\tau)\theta_4(z|\tau)},\\
\dn(u,k)&=\frac{\theta_4(0|\tau)\theta_3(z|\tau)}{\theta_3(0|\tau)\theta_4(z|\tau)},
\end{align*}
where $k$ is as above and $u=\theta_3(0|\tau)^2 z$. Prove that the trigonometric limits of these functions are
$$\sn(u,0)=\sin(u),\qquad \cn(u,0)=\cos(u),\qquad \dn(u,0)=1. $$
Using, for instance, Weierstrass's identity \eqref{ttr}, show that 
$$\sn(u,k)^2+\cn(u,k)^2=k^2\sn(u,k)^2+\dn(u,k)^2=1.$$
Note that 
$$\sn(u,k)^2=C\frac{\theta(e^{\pm 2i z};q^2)}{\theta(qe^{\pm 2iz};q^2)}$$
is of the form \eqref{xe} up to a change of variables. As in Exercise \ref{ele}, show that $y(u)=\sn(u,k)^2$ satisfies
the differential equation
$$(y')^2=4y(1-y)(1-k^2y).$$
Deduce that\footnote{At least formally, it follows that the inverse of $u\mapsto \sn(u,k)$ is given by
$$\sn^{-1}(x)=\int_0^x\frac{dt}{\sqrt{(1-t^2)(1-k^2t^2)}}. $$
This is known as an elliptic integral of the first kind. The related elliptic integral of the second kind,
$$\int_0^x{\sqrt{\frac{1-k^2t^2}{1-t^2}}}\,dt, $$
appears when one tries to compute the arc-length of an ellipse. This is the somewhat far-fetched historical reason for the terminology "elliptic function".
}
$$\frac d{du}\sn(u,k)=\sqrt{(1-\sn(u,k)^2)(1-k^2\sn(u,k)^2)}=\cn(u,k)\dn(u,k). $$
\end{problem}

\chapter{Elliptic hypergeometric functions}\label{ehc}

\section{Three levels of hypergeometry}\label{his}

A \emph{classical hypergeometric series} is a series $\sum_{k}c_k$ such that $c_{k+1}/c_k$ is a rational function of $k$. 
Examples include all the standard Taylor series encountered in calculus. For instance,
\begin{equation}\label{ts}e^z=\sum_{k=0}^\infty \frac{z^k}{k!},\qquad \sin z=\sum_{k=0}^\infty\frac{(-1)^k z^{2k+1}}{(2k+1)!},\qquad
\arctan z =\sum_{k=0}^\infty\frac{(-1)^kz^{2k+1}}{2k+1} \end{equation}
have  termwise ratio
$$\frac{z}{k+1},\qquad -\frac{z^2}{(2k+2)(2k+3)},
\qquad-\frac{(2k+1)z^2}{2k+3},
  $$
  respectively.
Hypergeometric series can also be finite; an example is the binomial sum
$$\sum_{k=0}^n\binom nk z^k, $$
with termwise ratio
\begin{equation}\label{tre}\frac{(n-k)z}{k+1}. \end{equation}
This also holds at the boundary of the summation range in the sense that, 
if we define $c_{n+1}=c_{-1}=0$, then the ratio
$c_{k+1}/c_k$ vanishes for $k=n$ and is infinite for $k=-1$, in agreement with \eqref{tre}.

If $\sum_k c_k$ is classical hypergeometric, we can factor the termwise quotient as
\begin{equation}\label{tqf}f(k)=\frac{c_{k+1}}{c_k}=z\frac{(a_1+k)\dotsm(a_r+k)}{(b_1+k)\dotsm (b_{s+1}+k)}. \end{equation}
As was just explained,
if we want to consider sums supported on $k\geq 0$, it is natural to assume that $f$ has a pole at $-1$. Thus, we take $b_{s+1}=1$. This is no restriction, as we  recover the general case if in addition $a_r=1$. Iterating \eqref{tqf} then gives
$$c_k=c_0\frac{(a_1)_k\dotsm(a_r)_k}{k!(b_1)_k\dotsm (b_s)_k}\,z^k, $$
where
$$(a)_k=a(a+1)\dotsm(a+k-1).$$ 
Thus, any classical hypergeometric series, supported on $k\geq 0$, is a constant multiple of
$${}_rF_s\left(\begin{matrix}a_1,\dots,a_r\\ b_1,\dots,b_s\end{matrix};z\right)=\sum_{k=0}^\infty\frac{(a_1)_k\dotsm(a_r)_k}{k!(b_1)_k\dotsm (b_s)_k}\,z^k.  $$
If $a_r=-n$ is a non-negative integer, this reduces to 
$${}_rF_s\left(\begin{matrix}a_1,\dots,a_{r-1},-n\\ b_1,\dots,b_s\end{matrix};z\right)=\sum_{k=0}^n\frac{(a_1)_k\dotsm(a_{r-1})_k(-n)_k}{k!(b_1)_k\dotsm (b_s)_k}\,z^k,  $$
which is the general form of a finite hypergeometric sum.

In the 19-th century, it became apparent that there is a  natural generalization of classical hypergeometric series called
\emph{basic hypergeometric series}. For these, the ratio $c_{k+1}/c_k$ is a rational function of $q^k$ for some fixed $q$ (known as the base).
If we write $q=e^{2\ti\pi\eta}$, the quotient can be expressed as a trigonometric function of $\eta$.
Thus, another possible name would be ``trigonometric hypergeometric series''.

In  the late 1980's, mathematical physicists discovered the even more general  \emph{elliptic hypergeometric series} while studying Baxter's elliptic solid-on-solid model \cite{d}. For these, the termwise ratio $c_{k+1}/c_k$ is an additively elliptic function of $k$. Partly because of convergence issues, the theory of infinite elliptic hypergeometric series is not well developped, so we will focus on finite sums. If one wants to consider elliptic extensions of infinite series, then it is usually  better to define them by integrals, see \S \ref{eis}.

The restriction to finite sums causes a slight problem of terminology. Namely, given any finite sum
$\sum_{k=0}^nc_k$, one can obviously find a rational function $f$ assuming the finitely many values $f(k)=c_{k+1}/c_k$. 
The same is true for trigonometric and elliptic functions. So it would seem that \emph{all} finite sums are both hypergeometric, basic hypergeometric and elliptic hypergeometric. Although this is correct in principle, it is never a problem in practice. 
The sums we will consider 
 depend on parameters, including $n$, in a way which will make the
 elliptic nature of the termwise ratio obvious. In particular, the degree of the relevant elliptic function (number of zeroes and poles up to periodicity) is independent of $n$.

\begin{problem}
Express the Taylor series \eqref{ts} in hypergeometric notation.
\end{problem}

\section{Elliptic hypergeometric sums}

In this section, we discuss  elliptic hypergeometric sums in general. 
As we explain in \S \ref{wps}, this is \emph{not} a natural class of functions. Sums appearing  
in practice, because of relations to physical models and/or interesting mathematical properties, are of more restricted types, such as the
 very-well-poised sums defined in \S \ref{wps}.

We want to consider sums $\sum_{k=0}^n c_k$ such that $c_{k+1}/c_k=f(k)$ 
for some additively elliptic function $f$. We denote the periods $1/\eta$ and $\tau/\eta$
and introduce the parameters $p=e^{2\ti\pi\tau}$, $q=e^{2\ti\pi\eta}$. We may then write
 $f(k)=g(q^k)$, where $g(px)=g(x)$. As usual, we assume that $0<|p|<1$. 
To avoid some (potentially interesting) degenerate situations, we will in addition assume 
that $1$, $\eta$ and $\tau$ are linearly independent over $\mathbb Q$ or,
 equivalently,
\begin{equation}\label{ndpq}q^k\neq p^l,\qquad k,\,l\in \mathbb Z,\quad (k,l)\neq (0,0).\end{equation}

By Theorem \ref{eftt}, we can write
\begin{equation}\label{cr}f(k)=\frac{c_{k+1}}{c_k}=z\frac{\theta(a_1q^k,\dots,a_{m+1}q^k;p)}{\theta(b_1q^k,\dots,b_{m+1}q^k;p)},\qquad a_1\dotsm a_{m+1}=b_1\dotsm b_{m+1}.\end{equation}
Let us introduce the \emph{elliptic shifted factorial}
$$(a;q,p)_k=\prod_{j=0}^{k-1}\theta(aq^j;p), $$
for which we employ the condensed notation
$$(a_1,\dots,a_m;q,p)_k=(a_1;q,p)_k\dotsm(a_m;q,p)_k, $$
\begin{equation}\label{efpm}(ax^\pm;q,p)_k=(ax;q,p)_k(a/x;q,p)_k.\end{equation}
Then, iterating \eqref{cr} gives
$$c_k=c_0\frac{(a_1,\dots,a_{m+1};q,p)_k}{(b_1,\dots,b_{m+1};q,p)_k}\,z^k. $$
Similarly as for classical hypergeometric sums, it is natural to require that $f(n)=0$ and $f(-1)=\infty$. For this reason, we take $a_{m+1}=q^{-n}$ and $b_{m+1}=q$. This is no restriction, as we can recover the general case by choosing in addition $a_m=q$ and $b_m=q^{-n}$. 
We conclude that, up to an inessential prefactor, the general form of a finite elliptic hypergeometric sum is
\begin{equation}\label{gem}{}_{m+1}E_m\left(\begin{matrix}q^{-n},a_1,\dots,a_m\\b_1,\dots,b_m\end{matrix};q,p;z\right)=\sum_{k=0}^n\frac{(q^{-n},a_1,\dots,a_m;q,p)_k}{(q,b_1,\dots,b_m;q,p)_k}\,z^k, \end{equation}
where the parameters should satisfy the \emph{balancing condition}\footnote{The reader familiar with basic hypergeometric series may be uncomfortable with the position of $q$ in this identity. We will return to this point at the end of \S \ref{fts}.}
\begin{equation}\label{bc}q^{-n} a_1\dotsm a_{m}=q b_1\dotsm b_{m}.\end{equation}
Note that \eqref{ndpq} is equivalent to requiring that $(q;q,p)_k\neq 0$ for $k\geq 0$, so that the sum is well-defined for generic $b_j$.

\begin{problem}\label{epi}
Show that
\begin{align*}(a;q,p)_{n+k}&=(a;q,p)_n(aq^n;q,p)_k,\\
(a;q,p)_{n-k}&=\frac{(-1)^kq^{\binom k2}(a;q,p)_n}{(aq^{n-1})^k(q^{1-n}/a;q,p)_k},\\
(a;q,p)_k&=(-1)^ka^k q^{\binom k2}(q^{1-k}/a;q,p)_k,\\
(p^ma;q,p)_n&=\frac{(-1)^{mn}}{a^{mn}p^{n\binom m2} q^{m\binom n2}}\,(a;q,p)_n. \end{align*}
These identities are used routinely when manipulating elliptic hypergeometric series. 
\end{problem}

\begin{problem}
Show that 
\begin{multline*}{}_{m+1}E_m\left(\begin{matrix}q^{-n},a_1,\dots,a_m\\b_1,\dots,b_m\end{matrix};q,p;z\right)\\
=\frac{(-z)^n}{q^{\binom{n+1}2}}\frac{(a_1,\dots,a_m;q)_n}{(b_1,\dots,b_m;q)_n}\,{}_{m+1}E_m\left(\begin{matrix}q^{-n},q^{1-n}/b_1,\dots,q^{1-n}/b_m\\q^{1-n}/a_1,\dots,q^{1-n}/a_m\end{matrix};q,p;\frac 1z\right). \end{multline*}
\end{problem}

\begin{problem}\label{mhp}
Let $\left(\begin{smallmatrix}a&b\\c&d\end{smallmatrix}\right)\in\mathrm{SL}(2,\mathbb Z)$. Using \eqref{tmt}, prove that 
$${}_{m+1}E_m\left(\begin{matrix}q^{-n},a_1,\dots,a_m\\b_1,\dots,b_m\end{matrix};q,p;z\right)
=
{}_{m+1}E_m\left(\begin{matrix}\tilde q^{-n},\tilde a_1,\dots,\tilde a_m\\ \tilde b_1,\dots,\tilde b_m\end{matrix};\tilde q,\tilde p;\tilde z\right),$$
where the parameters are related by
\begin{align*}a_j&=e^{2\ti\pi \eta\alpha_j},& b_j&=e^{2\ti\pi\eta \beta_j}, & p&=e^{2\ti\pi\tau}, & q&=e^{2\ti\pi\eta},\\
\tilde a_j&=e^{2\ti\pi \eta\alpha_j/(c\tau+d)},& \tilde b_j&=e^{2\ti\pi\eta \beta_j/(c\tau+d)}, & \tilde p&=e^{2\ti\pi(a\tau+b)/(c\tau+d)}, & \tilde q&=e^{2\ti\pi\eta/(c\tau+d)},
\end{align*}
$$\tilde z=\exp\left(\frac{\ti \pi c\eta^2}{c\tau+d}\left(1-n^2+\sum_{j=1}^{m}\left(\beta_j^2-\alpha_j^2\right)\right) \right)z $$
and we assume\footnote{If in addition  
$n^2+\sum_{j}\alpha_j^2=1+\sum_{j}\beta_j^2,$
 then $\tilde z=z$. In this situation, ${}_{m+1}E_m$
  is known as a \emph{modular} hypergeometric sum.} (without loss of generality in view of \eqref{bc})
$$-n+\sum_{j=1}^m \alpha_j=1+\sum_{j=1}^m\beta_j.$$
\end{problem}

\section{The Frenkel--Turaev sum}\label{fts}

Disregarding for a moment the condition \eqref{bc}, the case $p=0$ of \eqref{gem} is the basic hypergeometric sum
\begin{equation}\label{bhs}
{}_{m+1}\phi_m\left(\begin{matrix}q^{-n},a_1,\dots,a_m\\b_1,\dots,b_m\end{matrix};q;z\right)=\sum_{k=0}^n\frac{(q^{-n},a_1,\dots,a_m;q)_k}{(q,b_1,\dots,b_m;q)_k}\,z^k, \end{equation}
where
$$(a;q)_k=(a;q,0)_k=(1-a)(1-aq)\dotsm (1-a q^{k-1}).  $$
Browsing the standard textbook \cite{gr}, one will find a large (perhaps overwhelming) number of identities for such sums, such as  the $q$-Saalsch\"utz summation
\begin{equation}\label{qs}{}_3\phi_2 \left(\begin{matrix}q^{-n},a,b\\c,abc^{-1}q^{1-n}\end{matrix};q;q\right)
=\frac{(c/a,c/b;q)_n}{(c,c/ab;q)_n}\end{equation}
and the Jackson summation
\begin{multline}\label{js}{}_8\phi_7 \left(\begin{matrix}a,q\sqrt a,-q\sqrt a,q^{-n},b,c,d,e\\ \sqrt a,-\sqrt a,aq^{n+1},aq/b,aq/c,aq/d,aq/e\end{matrix};q;q\right)\\
=\frac{(aq,aq/bc,aq/bd,aq/cd;q)_n}{(aq/b,aq/c,aq/d,aq/bcd;q)_n}, \qquad a^2q^{n+1}=bcde.\end{multline}
If we let $a$, $d$ and $e$ tend to zero in \eqref{js}, in such a way that 
$a/d$ and $a/e$ is fixed, we recover \eqref{qs}. 
In fact, \eqref{js} is a top result in a hierarchy of  summations, containing \eqref{qs} and many
other results as limit cases. 
In \cite{gr}, the theory of basic hypergeometric series is built from bottom up, starting with
the simplest results like the $q$-binomial theorem, using these to prove intermediate results like \eqref{qs} and eventually proceeding to \eqref{js} and beyond. For elliptic hypergeometric series, such an approach is impossible
since, as a rule of thumb, only ``top level results'' extend to the elliptic level.
This may  be the reason why it took so long before elliptic hypergeometric functions were discovered.

Let us return to \eqref{qs} and explain why it does not exist at the elliptic level. 
As is emphasized in Ismail's lectures at the present summer school \cite{i}, it is natural to view \eqref{qs} as a connection formula for the ``Askey--Wilson monomials''
$$\phi_n(y;a)=(a x,a/x;q)_n,\qquad y=\frac{x+x^{-1}}2 $$
(equivalently, $x=e^{\ti\theta}$ and $y=\cos\theta$).
Indeed, if we replace $a$ by $ax$ and $b$ by $a/x$,  \eqref{qs} takes the form
$$\phi_n(y;c/a)=\sum_{k=0}^n C_k\,\phi_k(y;a).$$
One might try to obtain an elliptic extension by considering
$$(cx^\pm/a;q,p)_n=\sum_{k=0}^n C_k (ax^\pm;q,p)_k. $$
However, such an expansion cannot exist, as the terms satisfy different quasi-periodicity conditions.
For instance, the case $n=1$ can be written 
$$C_0=\theta(cx^\pm/a;p)-C_1\theta(ax^\pm;p). $$
Here, 
both terms on the right satisfy $f(px)=f(x)/px^2$, wheras the left-hand side is independent of $x$. This  is absurd (except in the trivial
cases $c\in p^\mathbb Z$ and $c\in a^2 p^\mathbb Z$, when we may take $C_0=0$).

Let us now consider \eqref{js} as a connection formula. Replacing  $b$ by $bx$ and  $c$ by $b/x$, it can be written as
$$(aqx^\pm/bd;q)_n=\sum_{k=0}^n C_k (bx^\pm;q)_k(q^{-n}bx^\pm/a;q)_{n-k}. $$
This seems  more promising to extend to the elliptic level. Changing parameters, we will consider the expansion problem
\begin{equation}\label{ebt}h_n(x;a)=\sum_{k=0}^n C_k^n\, h_k(x;b)h_{n-k}(x;c), \end{equation}
where
$h_n(x;a)=(ax^\pm;q,p)_n$.
Here, each term  satisfies the same quasi-perodicity relation. 
Following \cite{res}, we will 
compute the coefficients in \eqref{ebt} by induction on $n$, in analogy with a standard proof of the binomial theorem. 

When $n=0$,  \eqref{ebt} holds trivially with $C_0^0=1$. Assuming that \eqref{ebt} holds for fixed $n$, we write
$$h_{n+1}(x;a)=\theta(aq^n x^\pm;p)\sum_{k=0}^n C_k^n h_k(x;b)h_{n-k}(x;c). $$
To proceed, we must split
$$\theta(aq^nx^\pm;p)=B\theta(bq^kx^\pm;p)+C\theta(cq^{n-k}x^\pm;p). $$
By \eqref{ttr}, such a splitting exists for generic parameters and is given by
$$B=\frac{\theta(acq^{2n-k},aq^k/c;p)}{\theta(bcq^n,bq^{2k-n}/c;p)},\qquad C=-\frac{bq^{2k-n}\theta(abq^{n+k},aq^{n-k}/b;p)}{c\,\theta(bcq^n,bq^{2k-n}/c;p)} $$
(note that we do not need to remember the exact form of \eqref{ttr}; we can simply compute $B$ and $C$ by substituting  $x=cq^{n-k}$ and $x=bq^k$). 
Then, \eqref{ebt} holds with $n$ replaced by $n+1$ and 
\begin{equation}\label{ept}C_{k}^{n+1}=\frac{\theta(acq^{2n-k+1},aq^{k-1}/c;p)}{\theta(bcq^n,bq^{2k-2-n}/c;p)}\,C_{k-1}^n -\frac{bq^{2k-n}\theta(abq^{n+k},aq^{n-k}/b;p)}{c\,\theta(bcq^n,bq^{2k-n}/c;p)}\,C_k^n.\end{equation}
This is an elliptic extension of Pascal's triangle.
We claim that the solution is the ``elliptic binomial coefficient''
\begin{equation}\label{ebc}C_k^n=q^{n-k}\frac{\theta(q^{2k-n}b/c;p)(q,ab,ac;q,p)_n(a/c;q,p)_k(a/b;q,p)_{n-k}}{\theta(b/c;p)(bc;q,p)_n(q,ab,qb/c;q,p)_k(q,ac,qc/b;q,p)_{n-k}}. \end{equation}
Indeed, plugging \eqref{ebc} into \eqref{ept} we obtain after cancellation 
\begin{multline*}\theta(q^{n+1},abq^n,acq^n,q^{2k-n-1}b/c;p)\\
=\theta(q^k,abq^{k-1},acq^{2n-k+1},q^kb/c;p)+q^k\theta(q^{n+1-k},abq^{n+k},acq^{n-k},q^{k-n-1}b/c;p). 
\end{multline*}
The reader should check that this is again an instance of \eqref{ttr}.

Plugging \eqref{ebc} into \eqref{ebt}, simplifying and changing parameters yields the following result, which is  the most fundamental and important fact about elliptic hypergeometric sums.
 It was first obtained by Frenkel and Turaev \cite{ft}, but with additional restrictions on the parameters it occurs somewhat implicitly in \cite{d}.

 \begin{theorem}\label{ftit}
 When $ a^2q^{n+1}=bcde$,
 we have the summation formula
\begin{multline}\label{ftsi}
\sum_{k=0}^n\frac{\theta(aq^{2k};p)}{\theta(a;p)}\frac{(a,q^{-n},b,c,d,e;q,p)_k}{(q,aq^{n+1},aq/b,aq/c,aq/d,aq/e;q,p)_k}\,q^k\\
\begin{split}&={}_{10}E_9 \left(\begin{matrix}a,q\sqrt a,-q\sqrt a,q\sqrt{pa},-q\sqrt{pa},q^{-n},b,c,d,e\\ \sqrt a,-\sqrt a,\sqrt{pa},-\sqrt{pa},aq^{n+1},aq/b,aq/c,aq/d,aq/e\end{matrix};q,p;q\right)\\
&=\frac{(aq,aq/bc,aq/bd,aq/cd;q,p)_n}{(aq/b,aq/c,aq/d,aq/bcd;q,p)_n}.\end{split}\end{multline}
\end{theorem}

Here, we used Exercise \ref{tde} to write
\begin{align}\nonumber\frac{\theta(aq^{2k};p)}{\theta(a;p)}&=\frac{\theta(\sqrt a q^k,-\sqrt a q^k,\sqrt{pa} q^k,-\sqrt{pa} q^k;p)}{\theta(\sqrt a,-\sqrt a,\sqrt{pa},-\sqrt{pa};p)}\\
 \label{vwpf}&=\frac{(q\sqrt a ,-q\sqrt a ,q\sqrt{pa} ,-q\sqrt{pa} ;p,q)_k}{(\sqrt a,-\sqrt a,\sqrt{pa},-\sqrt{pa};p,q)_k}.
 \end{align}
 Note that the case $p=0$, 
$$\frac{1-aq^{2k}}{1-a}=\frac{(q\sqrt a,-q\sqrt a;q)_k}{(\sqrt a,-\sqrt a;q)_k}$$
contains two factors instead of four, so the ${}_{10}E_9$ reduces to the ${}_8\phi_7$ sum in \eqref{js}. This  explains an apparent discrepancy between 
the balancing conditions for basic and elliptic hypergeometric series. Namely,
the balancing condition for  \eqref{gem} is
$q^{-n} a_1\dotsm a_{m}=qb_1\dotsm b_{m}$, whereas \eqref{bhs} is called balanced if
$q^{1-n} a_1\dotsm a_{m}=b_1\dotsm b_{m}$. With these definitions, the left-hand sides of
\eqref{js} and \eqref{ftsi} are both balanced. The shift by $q^2$ comes from  the additional numerator parameters $\pm q\sqrt{pa} $ and denominator parameters $\pm\sqrt{pa}$, which become invisible in the trigonometric limit.

\begin{problem}
 Show that the case $n=1$ of \eqref{qs} is equivalent to the trigonometric case of \eqref{ta} and the case $n=1$ of \eqref{js} to the trigonometric case of \eqref{na}. (This gives another explanation for why
 the Saalsch\"utz summation does not exist at the elliptic level.)
\end{problem}

\begin{problem}
In \eqref{ftsi}, multiply $a$, $d$ and $e$ by $\sqrt p$ and then let $p\rightarrow 0$. Show that you obtain \eqref{qs} in the limit. (In this sense, the Saalsch\"utz summation \emph{does} exist at the elliptic level.)
\end{problem}

\begin{problem}
As we have remarked, \eqref{js} contains \eqref{qs} as a limit case. Why can't we take the corresponding limit of \eqref{ftsi} when $p\neq 0$?
\end{problem}

\begin{problem}\label{lie}
Show that (assuming \eqref{ndpq}) the functions $(h_k(x;a)h_{n-k}(x;b))_{k=0}^n$ are linearly independent for generic $a$ and $b$. Combining this with the dimension count in Exercise \ref{pxe}, deduce \emph{a priori}
that the expansion \eqref{ebt} exists generically.
\end{problem}

\begin{problem} Show that 
$$\lim_{q\rightarrow 1}\frac{(q^{-n};q,p)_k}{(q;q,p)_k}=(-1)^k\binom nk $$
and verify that the limit $q\rightarrow 1$ of  \eqref{ftsi}, all other parameters being fixed, is equivalent to the classical binomial theorem.
\end{problem}

\section{Well-poised and very-well-poised sums}\label{wps}

In concrete examples, the interesting variable in \eqref{gem} is not $z$; indeed, it is usually fixed to $z=q$. It is more fruitful to consider
\eqref{gem} as a function of the parameters $a_j$ and $b_j$. When viewed in this way, \eqref{gem} is not very natural,
as the terms have different quasi-periodicity. For instance, consider the behaviour under the simultaneous shift $a_1\mapsto pa_1$, $b_1\mapsto pb_1$ (which preserves \eqref{bc}). By 
 Exercise \ref{epi}, the $k$-th term in \eqref{gem} is then multiplied by $(b_1/a_1)^k$, which in general depends on $k$.
 By contrast, it is easy to see that each term in the sum
 $$\sum_{k=0}^n\frac{(x_0x_1,-x_0/x_1;q,p)_k}{(-x_0x_1,x_0/x_1;q,p)_k}\,z^k $$
 is invariant under either of the transformations $x_0\mapsto px_0$ and $x_1\mapsto px_1$.
 This is an example of what  Spiridonov \cite{st} calls a \emph{totally elliptic} hypergeometric sum.
 To describe all such 
 sums explicitly seems to be an open problem.
 We will be content with giving an interesting special case.
 
 \begin{proposition}\label{wl}
 The sum 
 \begin{equation}\label{ws}\sum_k\frac{(x_0x_1,\dots,x_0x_{m+1};q,p)_k}{(x_0/x_1,\dots,x_0/x_{m+1};q,p)_k}\,z^k, \end{equation}
 subject to the balancing condition $x_1^2\dotsm x_{m+1}^2=1$, is totally elliptic in the sense that each term is invariant under the simultaneous shifts
 $x_j\mapsto p^{\alpha_j}x_j$, where $\alpha_0,\dots,\alpha_{m+1}$ are integers subject to $\alpha_1+\dots+\alpha_{m+1}=0$. 
 \end{proposition}

 \begin{proof}
 By Exercise \ref{epi}, under such a shift the $k$-th term is multiplied by
 \begin{multline*}
 \prod_{j=1}^{m+1}\frac{(-1)^{\alpha_0-\alpha_j}(x_0/x_j)^{\alpha_0-\alpha_j}p^{k\binom{\alpha_0-\alpha_j}2}q^{(\alpha_0-\alpha_j)\binom k2}}{(-1)^{\alpha_0+\alpha_j}(x_0x_j)^{\alpha_0+\alpha_j}p^{k\binom{\alpha_0+\alpha_j}2}q^{(\alpha_0+\alpha_j)\binom k2}}\\
 =\prod_{j=1}^{m+1}\frac{1}{x_0^{2\alpha_j}x_j^{2\alpha_0}p^{k(2\alpha_0-1)\alpha_j}q^{2\alpha_j\binom k2}}=1.
 \end{multline*}
 \end{proof}
 
 Following \cite{st},
 we call  \eqref{ws} a \emph{well-poised} elliptic hypergeometric sum. Note that convergence is not an issue, as Proposition \ref{wl}
 is just a statement about the individual terms. 
 As a first step towards considering finite sums, suppose that the summation is
 over non-negative integers and that
  $x_0/x_{m+1}=q$.  
 Writing $a=x_0^2/q$ and $b_j=x_jx_0$ for $j=1,\dots,m$ we arrive at the series
 \begin{equation}\label{nwp}\sum_{k=0}^\infty\frac{(a,b_1,\dots,b_{m};q,p)_k}{(q,aq/b_1,\dots,aq/b_{m};q,p)_k}\,z^k, 
 \end{equation}
 where the balancing condition is $b_1^2\dotsm b_{m}^2=q^{m+1}a^{m-1}$.  
 
The Frenkel--Turaev sum \eqref{ftsi} is of the form \eqref{nwp}. Indeed,  the factor \eqref{vwpf} can alternatively be written 
 \begin{equation}\label{vwpa}\frac{\theta(aq^{2k};p)}{\theta(a;p)}=\frac{(-1)^k}{q^k}\frac{(q\sqrt a,-q\sqrt a,q\sqrt{a/p},-q\sqrt{ap};p,q)_k}{(\sqrt a,-\sqrt a,\sqrt{ap},-\sqrt{a/p};p,q)_k},\end{equation}
 which looks less symmetric but fits the well-poised pattern.
 Well-poised series containing the factor \eqref{vwpa} are called \emph{very-well-poised}. Clearly, any well-poised series can be considered as very-well-poised, since one can 
 artificially introduce this factor and then remove it again  by choosing $(b_1,\dots,b_4)=(\sqrt a,-\sqrt a,\sqrt{ap},-\sqrt{a/p})$. Typically, the $z$-variable in a very-well-poised series is fixed to $q$. 
 We will therefore use the notation
 $${}_{m+1}V_m(a;b_1,\dots,b_{m-4};q,p)=\sum_{k=0}^\infty\frac{\theta(aq^{2k};p)}{\theta(a;p)}\frac{(a,b_1,\dots,b_{m-4};q,p)_k}{(q,aq/b_1,\dots,aq/b_{m-4};q)_k}\,q^k.$$
 The balancing condition for this series is 
 $$b_1^2\dotsm b_{m-4}^2=q^{m-7}a^{m-5}. $$
As before, because of convergence problems one usually restricts to the case when $b_{m-4}= q^{-n}$, when $n$ is a non-negative integer.

\begin{problem}
Show that the sum
$$\sum_k\frac{(x_1^2,\dots,x_m^2;q,p)_k}{(-x_1^2,\dots,-x_m^2;q,p)_k}\,z^k $$
is totally elliptic in the sense of being invariant under any shift $x_j\mapsto px_j$. (This shows that not all totally elliptic series are well-poised.)
\end{problem}

\begin{problem}
Show that the well-poised series \eqref{ws} is modular in the sense explained in the footnote to Exercise \ref{mhp}.
\end{problem}

\begin{problem}
By reversing the order of summation, show that
$${}_{m+1}V_m(q^{-n};b_1,\dots,b_{m-4};q,p)=0 $$
if $n$ is even and the balancing condition holds.
\end{problem}

\section{The sum ${}_{12}V_{11}$}\label{tves}

Historically, the first examples of elliptic hypergeometric functions were sums of the form ${}_{12}V_{11}$ (see \cite{d} and \S \ref{fhs}).
Again following \cite{res},
we will see how such sums appear from a natural generalization of the expansion \eqref{ebt}.

Let us write
$C_k^n(a,b,c) $
for the coefficients (depending also on $p$ and $q$) given in \eqref{ebc} and appearing in the expansion
\begin{equation}\label{ebe}h_n(x;a)=\sum_{k=0}^nC_k^n(a,b,c)h_k(x;b)h_{n-k}(x;c). \end{equation} 
We will study more general coefficients
$R_k^l(a,b,c,d;n) $
appearing in the expansion 
\begin{equation}\label{ere}h_k(x;a)h_{n-k}(x;b)=\sum_{l=0}^n R_k^l(a,b,c,d;n) h_l(x;c)h_{n-l}(x;d).\end{equation}
By Exercise \ref{lie}, these coefficients exist uniquely for generic parameter values.

One can obtain explicit expressions for $R_k^l$ by iterating \eqref{ebe}. For instance, writing
\begin{multline*}h_k(x;a)h_{n-k}(x;b)
=\sum_{j=0}^k C^k_j(a,c,bq^{n-k})\,
h_j(x;c)h_{n-j}(x;b)\\
=\sum_{j=0}^k\sum_{m=0}^{n-j}
C^k_j(a,c,bq^{n-k}) C_{m}^{n-j}(b,cq^j,d)\,h_{j+m}(x;c)h_{n-j-m}(x;d)
\end{multline*}
gives
 $$R_k^l(a,b,c,d;n)=\sum_{j=0}^{\min(k,l)}C_j^k(a,c,bq^{n-k})C_{l-j}^{n-j}(b,cq^j,d). $$
Plugging in the expression
\eqref{ebc} and simplifying
one finds that 
\begin{multline}\label{rkl}
R_k^l(a,b,c,d;n)\\
\begin{split}&=q^{l(l-n)}\frac{(q;q,p)_n(ac,a/c;q,p)_k(q^{n-l}bd,b/d;q,p)_l
(bc,b/c;q,p)_{n-k}(b/c;q,p)_{n-l}}{(cd,b/c;q,p)_n(q,bc,q^{l-n}c/d;q,p)_l(q,q^{-l}d/c;q,p)_{n-l}}\\
&\quad
\times{}_{12}V_{11}(q^{-n}c/b;q^{-k},q^{-l},q^{k-n}a/b,q^{l-n}c/d,cd,q^{1-n}/ab,
qc/b;q,p).
\end{split}\end{multline}

The uniqueness of the expansion  \eqref{ere}  implies many non-trivial properties of ${}_{12}V_{11}$-sums.
For instance, it follows that
$$R_k^l(a,b,c,d;n)=R_{n-k}^l(b,a,c,d;n). $$
Writing this out explicitly, we get the  identity
\begin{multline}\label{pbt}
{}_{12}V_{11}(q^{-n}c/b;q^{-k},q^{-l},q^{k-n}a/b,q^{l-n}c/d,cd,q^{1-n}/ab,
qc/b;q,p)\\
\begin{split}&=\frac{(a/d,q^{n-l}ad,bc,q^{n-l}b/c;q,p)_l}{(b/d,q^{n-l}bd,ac,q^{n-l}a/c;q,p)_l}\\
&\quad\times{}_{12}V_{11}(q^{-n}c/a;q^{k-n},q^{-l},q^{-k}b/a,q^{l-n}c/d,cd,q^{1-n}/ab,
qc/a;q,p).
\end{split}\end{multline}
Note that the left-hand side is a sum supported on $[0,\min(k,l)]$ whereas the right-hand side is supported
on $[0,\min(n-k,l)]$. Thus, both sides still terminate if $k$ and $n$ are replaced by continuous parameters, as long as
$l$ is a non-negative integer. It is natural to ask whether \eqref{pbt} still holds in that case. Indeed, it does, which is
the content of the following result due to Frenkel and Turaev \cite{ft} in general and to Date et al.\ \cite{d} under some additional restrictions on the parameters.
As the case $p=0$ is due to Bailey, it is called the elliptic Bailey transformation.

\begin{theorem}\label{btp} Let $n$ be a non-negative integer, $bcdefg=q^{n+2}a^3$
and $\lambda=a^2q/bcd$. Then,
\begin{multline}\label{bt}{}_{12}V_{11}(a;b,c,d,e,f,g,q^{-n};q,p)=\frac{(aq,aq/ef,\lambda q/e,\lambda q/f;q,p)_n}{(aq/e,aq/f,\lambda q,\lambda q/ef;q,p_n}\\
\times {}_{12}V_{11}(\lambda;\lambda b/a,\lambda c/a,\lambda d/a,e,f,g,q^{-n};q,p). 
\end{multline}
\end{theorem}

In the following two exercises we give two proofs of Theorem \ref{btp}, one that uses \eqref{pbt} and one  that doesn't.

\begin{problem}\label{cbp}
Use \eqref{pbt} to show that \eqref{bt} holds if $b=q^{-k}$ and $d=aq^{N+1}$, where $k$ and $N$ are non-negative integers with $N\geq\max(k,n)$.
Then show that, after replacing $c$ by $c/bd$, both sides of \eqref{bt} are invariant under the independent substitutions $b\mapsto bp$ and $d\mapsto dp$. 
Finally, use analytic continuation to deduce \eqref{bt} in general.
\end{problem}

\begin{problem}\label{ibp}
Consider the double sum
\begin{multline*}\sum_{0\leq x\leq y\leq n}\frac{(a;q,p)_{x+y}(a/\lambda;q,p)_{y-x}}{(q\lambda;q,p)_{x+y}(q;q,p)_{y-x}}
\frac{\theta(\lambda q^{2x};p)}{\theta(\lambda;p)}\frac{(\lambda,\lambda b/a,\lambda c/a,\lambda d/a;q,p)_x}{(q,aq/b,aq/c,aq/d;q,p)_x}\left(\frac{aq}\lambda\right)^{x}\\
\times\frac{\theta(aq^{2y};p)}{\theta(a;p)}\frac{(q^{-n},e,f,g;q,p)_y}{(q^{n+1},aq/e,aq/f,aq/g;q,p)_y}\,q^{y}. 
\end{multline*}
Using \eqref{ftsi} to compute both the sum in $x$ and the sum in $y$, deduce \eqref{bt}.
\end{problem}

\begin{problem} By an argument similar to that in Exercise \ref{cbp}, deduce from Theorem \ref{btp} that, when $m\geq n$ are non-negative integers
and
 $efgh=a^2q^{1+n-m}$,
\begin{multline*}
{}_{12}V_{11}(a;q^{-n},c,aq^{m+1}/c,e,f,g,h;q,p)\\=(efg)^{m-n}\frac{(e/a,f/a,g/a,h/a;q,p)_{m-n}}{(a,ef/a,eg/a,fg/a;q,p)_{m-n}}\\
\times{}_{12}V_{11}(aq^{n-m};q^{-m},cq^{n-m},aq^{n+1}/c,e,f,g,h;q,p).
\end{multline*}
\end{problem}

\begin{problem}\label{isp}
In the previous problem, take $m=n+1$ and then let $c\rightarrow q$. Deduce the indefinite summation formula
\begin{multline*}\sum_{k=0}^n\frac{\theta(aq^{2k};p)}{\theta(a;p)}\frac{(e,f,g,h;q,p)_k}{(aq/e,aq/f,aq/g,aq/h;q,p)_k}q^k\\
= \frac{\theta(a/e,a/f,a/g,a/efg;p)}{\theta(a,a/ef,a/eg,a/fg;p)}\left(1-\frac{(e,f,g,h;q,p)_{n+1}}{(a/e,a/f,a/g,a/h;q,p)_{n+1}}\right),\quad efgh=a^2.
\end{multline*}
Finally, prove this identity directly by induction on $n$.\footnote{For generalizations to multi-basic series, see \eqref{wi} and Exercise \ref{qbp}.}
\end{problem}

\begin{problem}
By iterating \eqref{bt} show that, when $bcdefg=q^{n+2}a^3$,
\begin{multline*}{}_{12}V_{11}(a;b,c,d,e,f,g,q^{-n};q,p)
=\frac{g^n(aq,b,aq/cg,aq/dg,aq/eg,aq/fg;q,p)_n}{(aq/c,aq/d,aq/e,aq/f,aq/g,b/g;q,p)_n}\\
\times{}_{12}V_{11}(q^{-n}g/b;q^{-n}g/a,aq/bc,aq/bd,aq/be,aq/bf,g,q^{-n};q,p).
\end{multline*}
\end{problem}

\begin{problem}
Show that the coefficients $R_k^l$ satisfy the ``addition formula''
$$R_k^l(a,b,e,f;n)=\sum_{j=0}^n R_k^j(a,b,c,d;n)\,R_j^l(c,d,e,f;n).$$
\end{problem}

\begin{problem}
Show the ``convolution formulas''
\begin{multline*}
R_{k_1+k_2}^l(a,b,c,d;n_1+n_2)
=\sum_{l_1+l_2=l}R_{k_1}^{l_1}(aq^{\alpha k_2},bq^{\beta(n_2-k_2)},c,
d;n_1)\\
\times R_{k_2}^{l_2}(aq^{(1-\alpha)k_1},bq^{(1-\beta)(n_1-k_1)},
cq^{l_1},dq^{n_1-l_1};n_2)
\end{multline*} 
for all $\alpha,\beta\in\{0,1\}$.
\end{problem}

\section{Biorthogonal rational functions}\label{brs}

One  definition of \emph{classical orthogonal polynomial} is a system of orthogonal polynomials 
whose derivatives are again orthogonal polynomials. By a classical  result of Sonine, this means Jacobi, Laguerre and Hermite polynomials. 
During the 20th century, it was gradually realized that this definition is too restrictive, as there are further polynomials that share the main properties of Jacobi polynomials, though they are  related to difference equations rather than differential equations. This development culminated in the Askey and $q$-Askey schemes, which together contain a large number (42 as they are counted in \cite{ks}) of polynomial systems. All these polynomials are degenerate cases of two five-parameter families: the $q$-Racah polynomials and the Askey--Wilson polynomials. Spiridonov and Zhedanov \cite{sz} and Spiridonov \cite{scb} obtained seven-parameter elliptic extensions of  these systems. These are neither orthogonal nor polynomial; instead, they are biorthogonal rational functions.\footnote{In \cite{scb}, even more general biorthogonal non-rational functions are considered.}
 In general, two systems $(p_j)$ and $(q_j)$  are called \emph{biorthogonal}
if, with respect to some scalar product, $\langle p_j,q_k\rangle=0$ for $j\neq k$.  For the specific systems considered here, $q_j$ is obtained from $p_j$ by a
simple change of parameters.

A very satisfactory consequence of our 
construction of  ${}_{12}V_{11}$-functions as  overlap coefficients is that their discrete biorthogonality 
falls out immediately (we are still following \cite{res}). Indeed, it is clear  that the coefficients $R_k^l$  satisfy 
$$
\delta_{kl}=\sum_{j=0}^n R_k^j(a,b,c,d;n)\,R_j^l(c,d,a,b;n),\qquad 0\leq k,l\leq n.
$$
Let us make the substitutions
$$(a,b,c,d)\mapsto(\sqrt{c/f},q/d\sqrt{cf},a\sqrt{cf},q^{-n}\sqrt{cf}/a)$$
in this identity, and introduce  parameters $b$ and $e$ such that $ab=q^{-n}$ and $abcdef=q$.
We can then express the result in terms of the functions
\begin{multline}\label{wof}
r_k\left(X(x);a,b,c,d,e,f;q,p\right)
=\frac{(ab,ac,ad,1/af;q,p)_k}{(aq/e;q,p)_k}\\
\times{}_{12}V_{11}(a/e;ax,a/x,q/be,q/ce,q/de,q^k/ef,q^{-k};q,p).
\end{multline}
Here, the right-hand side is an even (in the multiplicative sense) elliptic function of $x$ and $X$ any generator for the field of such functions, see Proposition \ref{uec}.
Thus, $r_k$ is a rational function of its first argument. After simplification, we find that if  $ab=q^{-n}$ 
and $abcdef=q$ then
$$
\sum_{j=0}^n w_j\, r_k\left(X(aq^j);a,b,c,d,e,f;q\right) r_l\left(X(aq^j);a,b,c,d,f,e;q\right)
=C_k\,\delta_{kl},$$
where
$$ w_j=\frac{\theta(a^2q^{2j};p)}{\theta(a^2;p)}\frac{(a^2,ab,ac,ad,ae,af;q,p)_j}
{(q,aq/b,aq/c,aq/d,aq/e,aq/f;q,p)_j}\,q^j$$
and 
$$C_k=\frac{(a^2q,q/cd,q/ce,q/de;q,p)_n}{(aq/c,aq/d,aq/e,aq/cde;q,p)_n}
\frac{(q,ab,ac,ad,bc,bd,cd;q,p)_k}{(1/ef;q,p)_{k}}\frac{\theta(1/ef;p)}{\theta(q^{2k}/ef;p)}\,q^{-k}. $$
This is an elliptic extension of the orthogonality relation for $q$-Racah polynomials; see Exercise \ref{awe} for an extension of Askey--Wilson polynomials.

\begin{problem}\label{rse}
Using Theorem \ref{bt}, show that the function $r_k$ is symmetric in the parameters $(a,b,c,d)$.
\end{problem}

\section{A quadratic summation}

There are many  transformation formulas between classical hypergeometric series whose variables satisfy a quadratic relation.
Extensions of such results to basic hypergeometric series typically involve a mixture of $q$-shifted factorials $(b;q)_k$ and $(b;q^2)_k$. 
Accordingly, we call an elliptic hypergeometric sum \emph{quadratic} if it combines factors of the form $(b;q,p)_k$ and
$(b;q^2,p)_k$. In view of the identity $(b;q^2,p)_k=(\sqrt b,-\sqrt b,\sqrt{pb},-\sqrt{pb};q,p)_k$, this should be viewed  more  as a rule of thumb than a precise definition. 
 We will
derive just one quadratic summation formula in order to give some idea about the relevant methods.
All results of this section are due to Warnaar \cite{w}.

Our starting point is a telescoping sum
$$ \sum_{k=0}^n(A_{k+1}-A_k)=A_{n+1}-A_0.$$
Substituting
$A_k=B_0\dotsm B_{k-1} C_{k}\dotsm C_n $
gives
$$\sum_{k=0}^n B_0\dotsm B_{k-1}(B_{k}-C_{k})C_{k+1}\dotsm C_n=B_0\dotsm B_n-C_0\dotsm C_n. $$
In view of \eqref{ttr}, it is natural to choose
$$B_k=\theta(a_kd_k^\pm,b_kc_k^\pm;p),\qquad C_k=\theta(b_kd_k^\pm,a_kc_k^\pm;p), $$
since then 
$$B_k-C_k=\frac{a_k}{c_k}\,\theta(c_kd_k^\pm,b_ka_k^\pm;p). $$
This gives the theta function identity  
\begin{multline}\label{wti}\sum_{k=0}^n\frac{a_k}{c_k}\,\theta(c_kd_k^\pm,b_ka_k^\pm;p)
\prod_{j=0}^{k-1}\theta(a_jd_j^\pm,b_jc_j^\pm;p)\prod_{j=k+1}^n\theta(b_jd_j^\pm,a_jc_j^\pm;p)\\
=\prod_{j=0}^n\theta(a_jd_j^\pm,b_jc_j^\pm;p)-\prod_{j=0}^n \theta(b_jd_j^\pm,a_jc_j^\pm;p).\end{multline}

Consider now the special case when $a_j=a$ and $b_j=b$ are independent of $j$, but $c_j=cq^j$, $d_j=dr^j$ form geometric progressions.
After simplification, we obtain
\begin{multline}\label{wi}
\sum_{k=0}^n\frac{\theta(cdq^kr^k;p)\theta(dr^k/cq^k;p)}{\theta(cd;p)\theta(d/c;p)}\frac{(da^\pm;r,p)_k(cb^\pm;q,p)_k}{(rdb^\pm;r,p)_k(qca^\pm;q,p)_k}\,q^k\\
=\frac{\theta(ac^\pm,db^\pm;p)}{\theta(ab^\pm,dc^\pm;p)}\left(1-\frac{(cb^\pm ;q,p)_{n+1}(da^\pm ;r,p)_{n+1}}{(ca^\pm;q,p)_{n+1}(db^\pm;r,p)_{n+1}}\right).
\end{multline}
This is a \emph{bibasic} sum.
The case $q=r$ can be deduced from the elliptic Bailey transformation, see Exercise \ref{isp}.

Consider \eqref{wi}  when $a=r^{-n}/d$ and $b=1/d$.  If $n>0$, the right-hand side vanishes, whereas if
 $n=0$, the zero is cancelled by the factor $\theta(a/b;p)$. We conclude that
\begin{equation}\label{wis}\sum_{k=0}^n\frac{\theta(cdq^kr^k;p)\theta(dr^k/cq^k;p)}{\theta(cd;p)\theta(d/c;p)}\frac{(r^{-n},r^nd^2;r,p)_k(cd^\pm;q,p)_k}{(r,rd^2;r,p)_k(qr^ncd,qr^{-n}c/d;q,p)_k}\,q^k=\delta_{n0}.\end{equation}

Sums that evaluate to Kronecker's delta function are useful for deriving new summations from known ones.  Indeed, suppose that $\sum_{k} a_{kn}=\delta_{n0}$ and that  we can compute  $c_k=\sum_n a_{kn}b_n$ for some sequence $b_n$. Then, 
\begin{equation}\label{sck}\sum_k c_k=\sum_n\delta_{n0}b_n=b_0.\end{equation}
There are several ways to apply this idea to \eqref{wis}, but we are content with giving one example. We take $r=q^2$, that is,
\begin{equation}\label{qi}a_{kn}=\frac{\theta(cdq^{3k};p)\theta(dq^k/c;p)}{\theta(cd;p)\theta(d/c;p)}\frac{(q^{-2n},q^{2n}d^2;q^2,p)_k(cd^\pm;q,p)_k}{(q^2,q^2d^2;q^2,p)_k(q^{2n+1}cd,q^{1-2n}c/d;q,p)_k}\,q^k.\end{equation}
One may expect that $b_n$ contains the factor
$$\frac{(d^2;q^2,p)_n(d/c;q,p)_{2n}}{(q^2;q^2,p)_n(cdq;q,p)_{2n}}=\frac{(d^2,d/c,qd/c;q^2,p)_n}{(q^2,cdq,cdq^2;q^2,p)_n}, $$
as it combines nicely with the $n$-dependent factors from \eqref{qi} (cf.~Exercise~\ref{epi}). This looks like part of a well-poised series, so we try to match  the sum
$\sum_n a_{kn}b_n$ with the Frenkel--Turaev sum. In the case $k=0$ this is just $\sum_n b_n$, so we are led to take
$$b_n=\frac{\theta(d^2q^{4n};p)}{\theta(d^2;p)}\frac{(d^2,d/c,qd/c,e,f,q^{-2m};q^2,p)_n}{(q^2,cdq,cdq^2,d^2q^2/e,d^2q^2/f,d^2q^{2m+2};q^2,p)_n}\,q^{2n}, $$
with the balancing condition
$c^2d^2q^{2m+1}=ef$.
We now compute
\begin{multline*}\sum_{n=k}^m\frac{(q^{-2n},q^{2n}d^2;q^2,p)_k}{(q^{2n+1}cd,q^{1-2n}c/d;q,p)_k}\, b_n\\
=q^{\binom k2}(dq/c)^k\frac{\theta(d^2q^{4k};p)}{\theta(d^2;p)}\frac{(d^2;q^2,p)_{2k}(d/c;q,p)_k(e,f,q^{-2m};q^2,p)_k}{(cdq;q,p)_{3k}(d^2q^2/e,d^2q^2/f,d^2q^{2m+2};q^2,p)_k}\\
\times\sum_{n=0}^{m-k}\frac{\theta(d^2q^{4k+4n};p)}{\theta(d^2q^{4k};p)}\frac{(d^2q^{4k},q^kd/c,q^{k+1}d/c,eq^{2k},fq^{2k},q^{2k-2m};q^2,p)_nq^{2n}}{(q^2,cdq^{1+3k},cdq^{2+3k},d^2q^{2+2k}/e,d^2q^{2+2k}/f,d^2q^{2m+2k+2};q^2,p)_n}\\
=\frac{(c^2q,d^2q^2;q^2,p)_m(cdq/e;q,p)_{2m}}{(d^2q^2/e,c^2q/e;q^2,p)_m(cdq;q,p)_{2m}}\frac{(d/c;q,p)_k(e,f,q^{-2m};q^2,p)_{k}}{(c^2q;q^2,p)_{k}(cdq/e,cdq/f,cdq^{1+2m};q,p)_{k}},
  \end{multline*}
where the last step is the Frenkel--Turaev sum. Thus, \eqref{sck} takes the form
\begin{multline*}\sum_{k=0}^m\frac{\theta(cdq^{3k};p)\theta(dq^k/c;p)}{\theta(cd;p)\theta(d/c;p)}
\frac{(cd,c/d,d/c;q,p)_k(e,f,q^{-2m};q^2,p)_{k}}{(q^2,d^2q^2,c^2q;q^2,p)_{k}(cdq/e,cdq/f,cdq^{1+2m};q,p)_{k}}
\,q^k\\
 =
\frac{(d^2q^2/e,c^2q/e;q^2,p)_m(cdq;q,p)_{2m}}{(c^2q,d^2q^2;q^2,p)_m(cdq/e;q,p)_{2m}}
.\end{multline*}
After the change of variables
$(c,d,e,f,m)\mapsto(\sqrt{ab},\sqrt{a/b},c,d,n)$,
we obtain the following quadratic summation due to Warnaar.

\begin{proposition}\label{wqp}
For $cd=a^2q^{2n+1}$, 
\begin{multline*}\sum_{k=0}^{n}\frac{\theta(aq^{3k};p)}{\theta(a;p)}\frac{(a,b,q/b;q,p)_k(c,d,q^{-2n};q^2,p)_k}{(q^2,aq^2/b,abq;q^2,p)_k(aq/c,aq/d,aq^{1+2n};q,p)_k}\,q^k\\ 
=\frac{(aq;q,p)_{2n}(abq/c,aq^2/bc;q^2,p)_n}{(aq/c;q,p)_{2n}(abq,aq^2/b;q^2,p)_n}.\end{multline*}
\end{proposition}

\begin{problem}
Show that the case $p=0$ of \eqref{wti}, after substituting  $a_k+a_k^{-1}\mapsto a_k$ and so on, takes the form
\begin{multline}\label{rti} \sum_{k=0}^n(c_k-d_k)(b_k-a_k)\prod_{j=0}^{k-1}(a_j-d_j)(b_j-c_j)\prod_{j=k+1}^n(b_j-d_j)(a_j-c_j)\\
=\prod_{j=0}^n(a_j-d_j)(b_j-c_j)-\prod_{j=0}^n(b_j-d_j)(a_j-c_j).
\end{multline}
Conversely, show that substituting $a_k\mapsto X(a_k)$ and so on in \eqref{rti}, where $X$ is as in \eqref{xe}, gives back the general case of \eqref{wti}.
\end{problem}

\begin{problem}\label{qbp}
Write down the  summation obtained from \eqref{wti} by choosing all four parameter sequences as independent geometric progressions (see \cite{gs}).
\end{problem}

\begin{problem}
Prove Proposition \ref{wqp} by the method of \S \ref{fts}.  
\end{problem}

\begin{problem}
Find a cubic summation formula by combining the Frenkel--Turaev sum with the case $r=q^3$ of \eqref{wis}  (see \cite[Thm.\ 4.5]{w}).
\end{problem}

\begin{problem}
Find a transformation formula that generalizes Proposition \ref{wqp} (see \cite[Thm.\ 4.2]{w}).
\end{problem}

\begin{problem}
Let $A$ and $B$ be lower-triangular matrices (the size is irrelevant) with entries
$$A_{ij}=\frac{\prod_{k=j}^{i-1}\theta(y_jz_k^\pm;p)}{\prod_{k=j+1}^i\theta(y_jy_k^\pm;p)},\qquad
B_{ij}=\frac{y_i\theta(y_jz_j^\pm;p)\prod_{k=j+1}^i\theta(y_iz_k^\pm;p)}{y_l\theta(y_iz_i^\pm;p)\prod_{k=j}^{i-1}\theta(y_iy_k^\pm;p)}.
 $$
Show that $B=A^{-1}$. Indeed, show that the identity $AB=I$ is equivalent to \eqref{vdf} and that the identity $BA=I$ 
is equivalent to the case $a_j\equiv d_n$, $b_j\equiv d_0$ of \eqref{wti}. As any left inverse is a right inverse, this gives an alternative proof
of \eqref{wis} as a consequence of \eqref{vdf}.
\end{problem}

\section{An elliptic Minton summation}

Minton  found the summation formula \cite{m}
\begin{equation}\label{mis}{}_{r+2}F_{r+1}\left(\begin{matrix}-n,b,c_1+m_1,\dots,c_r+m_r\\b+1,c_1,\dots,c_r\end{matrix};1\right) =
\frac{n!(c_1-b)_{m_1}\dotsm(c_r-b)_{m_r}}{(b+1)_n(c_1)_{m_1}\dotsm(c_r)_{m_r}},\end{equation}
where $m_i$ and $n$ are non-negative integers such that $m_1+\dots+ m_r\leq n$. This was extended to non-terminating series by Karlsson, so sums with  integral parameter differences 
are often referred to as Karlsson--Minton-type.

Following \cite{rs}, we will obtain an elliptic extension of \eqref{mis} from the elliptic partial fraction expansion \eqref{dpf}.
We first replace $n$ by $n+1$ in \eqref{dpf} and rewrite it as
\begin{equation}\label{spf}\sum_{j=0}^n\frac {\prod_{k=1}^{n}\theta(y_jz_k^{\pm};p)}{\theta(xy_j^\pm;p)\prod_{k=0,\,k\neq j}^n\theta(y_jy_k^\pm;p)}=\frac{\prod_{k=1}^{n}\theta(xz_k^\pm;p)}{\prod_{k=0}^n\theta(x y_k^\pm;p)}. \end{equation}
We now specialize $y$ to be  a geometric progression and $z$ to be  a union of geometric progressions.\footnote{As the latter progressions may have length $1$, the case of general $z$ is included.}
That is, we write $y_j=aq^{j}$ and 
$$(z_1,\dots,z_n)=(c_1,c_1q,\dots,c_1q^{m_1},\dots,c_r,c_rq,\dots,c_rq^{m_r}), $$
where $m_j$ are non-negative integers summing to $n$. Then, the  left-hand side of \eqref{spf} takes the form
\begin{multline*}\sum_{j=0}^n\frac{\prod_{l=1}^r\prod_{k=1}^{m_l} \theta(ac_lq^{j+k-1},aq^{1+j-k}/c_l;p)}{\theta(xaq^j,xq^{-j}/a;p)\prod_{k=0,\,k\neq j}^n\theta(a^2q^{j+k},q^{j-k};p)}\\
\begin{split}&=
\sum_{j=0}^n\frac{\prod_{l=1}^r(ac_lq^{j},aq^{1+j-m_l}/c_l;q,p)_{m_l}}{\theta(xaq^j,xq^{-j}/a;p)(q,a^2q^j;q,p)_j(q^{j-n},a^2q^{2j+1};q,p)_{n-j}}\\
&=\frac{\prod_{l=1}^r(ac_l,aq^{1-m_l}/c_l;q)_{m_l}}{\theta(xa^\pm;p)(q^{-n},a^2q;q,p)_n}\\
&\quad\times\sum_{j=0}^n
\frac{\theta(a^2q^{2j};p)}{\theta(a^2;p)}\frac{(a^2,ax^\pm,q^{-n};q,p)_j}{(q,aqx^\pm, a^2q^{n+1};q,p)_j}\,q^j\prod_{l=1}^r\frac{(ac_lq^{m_l},aq/c_l;q,p)_j}
{(ac_l,aq^{1-m_l}/c_l;q,p)_j}.
\end{split}\end{multline*}
and the right-hand side is
$$\frac{\prod_{l=1}^r\prod_{k=1}^{m_l} \theta(xc_lq^{k-1},xq^{1-k}/c_l;p)}{\prod_{k=0}^n\theta(xaq^k,xq^{-k}/a;p)}=
\frac{\prod_{l=1}^r(xc_l,xq^{1-m_l}/c_l;q,p)_{m_l}}{(xa,xq^{-n}/a;q,p)_{n+1}}.$$
After the change of variables $a\mapsto\sqrt a$, $x\mapsto b/\sqrt a$, $c_l\mapsto c_l/\sqrt a$, we arrive at the following 
elliptic extension of Minton's identity.

\begin{proposition}\label{kmp}
If $m_1,\dots,m_r$ are non-negative integers and $n=m_1+\dots+m_r$, 
\begin{multline*}
{}_{2r+8}V_{2r+7}(a;b,a/b,q^{-n},c_1q^{m_1},\dots,c_rq^{m_r},aq/c_1,\dots,aq/c_r;q,p)\\
\begin{split}&=\sum_{j=0}^n
\frac{\theta(aq^{2j};p)}{\theta(a;p)}\frac{(a,b,a/b,q^{-n};q,p)_j}{(q,bq,aq/b;q,p)_j}\,q^j
\prod_{l=1}^r\frac{(c_lq^{m_l},aq/c_l;q,p)_j}
{(c_l,aq^{1-m_l}/c_l;q,p)_j}\\
&=\frac{(q,aq;q,p)_n}{(aq/b,bq;q,p)_n}\prod_{l=1}^r\frac{(c_l/b,bc_l/a;q,p)_{m_l}}{(c_l,c_l/a;q,p)_{m_l}}.
\end{split}\end{multline*}
\end{proposition}

Proposition \ref{kmp} may seem to generalize only the case $\sum_j m_j=n$ of \eqref{mis}, but it should be noted that if we start from that special case and let $c_r\rightarrow \infty$ we obtain the general case.

\begin{problem}
Prove a transformation for Karlsson--Minton type series by starting from \eqref{vdf} and specializing $y$ to a union of two geometric progressions (see \cite[Cor.\ 4.5]{rs}).
\end{problem}

\begin{problem}
Deduce a Karlsson--Minton-type summation formula from \eqref{epf} (the result is less attractive than Proposition \ref{kmp}; see \cite[Cor.\ 5.3]{rs}). 
\end{problem}

\section{The elliptic gamma function}

The classical gamma function satisfies $\Gamma(z+1)=z\,\Gamma(z)$, which upon iteration gives 
$(z)_n={\Gamma(z+n)}/{\Gamma(z)}$.
 We will need an elliptic analogue $\Gamma(x;q,p)$ of the gamma function, which was
introduced by Ruijsenaars \cite{ru}.  It is natural to demand that  
\begin{equation}\label{egf}\Gamma(qx;q,p)=\theta(x;p)\Gamma(x;q,p),\end{equation}
 which upon iteration gives
\begin{equation}\label{egp}(x;q,p)_n=\frac{\Gamma(q^nx;q,p)}{\Gamma(x;q,p)}. \end{equation}
To solve \eqref{egf}, consider first in general a functional equation of the form $f(qx)=\phi(x)f(x)$. Upon iteration, it gives
$$f(x)=\frac{1}{\phi(x)\phi(qx)\dotsm\phi(q^{N-1} x)}\,f(q^Nx) .$$
If  $|q|<1$ and $\phi(x)\rightarrow 1$ quickly enough as $x\rightarrow 0$, one solution will be
$$f(x)=\frac{1}{\prod_{k=0}^\infty\phi(q^kx)}. $$
Alternatively, we can iterate the functional equation in the other direction and obtain
$$f(x)=\phi(q^{-1}x)\phi(q^{-2}x)\dotsm\phi(q^{-N}x)f(q^{-N}x). $$
In this case, if $\phi(x)\rightarrow 1$ quickly as $|x|\rightarrow\infty$, we find the solution
$$f(x)=\prod_{k=0}^\infty\phi(q^{-k-1}x). $$
In the case at hand, $\phi(x)=\theta(x;p)$, we can write $\phi(x)=\phi_1(x)\phi_2(x)$, where $\phi_1(x)=(x;p)_\infty$ and $\phi_2(x)=(p/x;p)_\infty$ tend to $1$ as $x\rightarrow 0$ and $|x|\rightarrow\infty$, respectively. This suggests the definition
$$\Gamma(x;q,p)=\prod_{k=0}^\infty\frac{\phi_2(q^{-k-1}x)}{\phi_1(q^k x)}=\prod_{j,k=0}^\infty\frac{1-p^{j+1}q^{k+1}/x}{1-p^jq^kx}.  $$
In a similar way as Lemma \ref{qpl}, one can then prove the following result. (In view of the symmetry between $p$ and $q$, 
we will from now on write $\Gamma(x;p,q)$ rather than $\Gamma(x;q,p)$.)

\begin{lemma}
The elliptic gamma function $\Gamma(x;p,q)$ is meromorphic as a function of $x\in\mathbb C^\ast$ and 
 of $p$ and $q$ with  $|p|$, $|q|<1$, with zeroes precisely at the points $x=p^{j+1}q^{k+1}$ and
poles precisely at the points $x=p^{-j}q^{-k}$, where $j,\,k\in\mathbb Z_{\geq 0}$. 
Moreover, \eqref{egf} holds.
\end{lemma}

Note the inversion formula
\begin{equation}\label{gi}\Gamma(x;p,q)\Gamma(pq/x;p,q)=1.\end{equation}
Just as for theta functions and elliptic shifted factorials, we will use the condensed notation
\begin{align*}\Gamma(x_1,\dots,x_m;p,q)&=\Gamma(x_1;p,q)\dotsm\Gamma(x_m;p,q), \\
\Gamma(ax^\pm;p,q)&=\Gamma(ax;p,q)\Gamma(a/x;p,q).\end{align*}

One would expect that the elliptic gamma function degenerates to the classical gamma function if we first take the trigonometric limit $p\rightarrow 0$ and then the rational limit $q\rightarrow 1$. Indeed, we have
$\Gamma(x;q,0)=1/(x;q)_\infty$ and \cite[\S 1.10]{gr}
$$\lim_{q\rightarrow 1}\,(1-q)^{1-x}\frac{(q;q)_\infty}{(q^x;q)_\infty}=\Gamma(x). $$

\begin{problem}
Show that the meromorphic solutions to $f(qx)=\theta(x;p)f(x)$ are precisely the functions  $f(x)=\Gamma(x;q,p)g(x)$, where $g$ is an arbitrary multiplicatively elliptic function
with period $q$. 
\end{problem}

\begin{problem} Give an analogue of the reflection formula 
$$\Gamma(z)\Gamma(1-z)=\frac{\pi}{\sin(\pi z)}$$ 
for  $\Gamma(x,q/x;p,q)$.
\end{problem}

\begin{problem} Give an analogue of the duplication formula 
$$\sqrt{\pi}\,\Gamma(2z)=2^{2z-1}\Gamma(z)\Gamma(z+1/2)$$
 for the function $\Gamma(x^2;p,q)$.
\end{problem}

\section{Elliptic hypergeometric integrals}\label{eis}

As was mentioned in \S \ref{his}, if we want to consider elliptic analogues of infinite hypergeometric series, it is often better to define them by integrals.  
The study of such integrals was initiated by Spiridonov, see e.g.\ \cite{sp1,scb}.

A model result for converting series to integrals is Barnes's integral representation
\begin{equation}\label{bi}{}_2F_1\left(\begin{matrix}a,b\\c\end{matrix};z\right)=\frac{\Gamma(c)}{\Gamma(a)\Gamma(b)}\int_{-\ti\infty}^{\ti\infty}\frac{\Gamma(a+s)\Gamma(b+s)\Gamma(-s)}{\Gamma(c+s)}\,(-z)^s\,\frac{ds}{2\pi\ti}. 
\end{equation}
This holds for $z\notin\mathbb R_{\geq 0}$ and $a,\,b\notin\mathbb Z_{<0}$. The integrand has poles at $s=-a-n$, $s=-b-n$
and $s=n$, where $n\in\mathbb Z_{\geq 0}$. The contour of integration should pass to the right of the first two sequences of poles but to the left of the third sequence. 

To prove \eqref{bi} one needs to know that the residue of the gamma function at $-n$ is $(-1)^n/n!$.  Consequently, the residue of the integrand at $n$ is
$$\frac 1{2\pi\ti}\frac{\Gamma(c)}{\Gamma(a)\Gamma(b)}\frac{\Gamma(a+n)\Gamma(b+n)}{\Gamma(c+n)n!}\,z^n
=\frac 1{2\pi\ti}\frac{(a)_n(b)_n}{(c)_nn!}\,z^n. $$ 
Thus, \eqref{bi} simply means that the integral is $2\pi\ti$ times the sum of all residues to the right of the contour. 
This follows from
Cauchy's residue theorem together with an estimate of the integrand that we will not go into.

In analogy with Barnes's integral, we will now consider a class of integrals closely related to  very-well-poised elliptic hypergeometric series \cite{scb}.
They have the form
\begin{equation}\label{wpi}I(t_0,\dots,t_m;p,q)=\int\frac{\Gamma(t_0x^\pm,\dots, t_mx^\pm;p,q)}{\Gamma(x^{\pm 2};p,q)}\frac{dx}{2\pi\ti x}. \end{equation}
Note that the factor
$$\frac1{\Gamma(x^{\pm2};p,q)}=-\frac{\theta(x^2;p)\theta(x^2;q)}{x^2} $$
is analytic for $x\neq 0$. Thus, the only poles of the integrand are at
\begin{equation}\label{ip}x=p^jq^kt_l,\qquad j,k\in\mathbb Z_{\geq 0}, \quad l=0,\dots,m \end{equation}
and at the reciprocal of these points. The integration is over a closed positively oriented contour such that the poles \eqref{ip} are inside the contour and
their reciprocals are outside. Such a contour exists if $t_jt_k\notin p^{\mathbb Z_{\leq 0}}q^{\mathbb Z_{\leq 0}}$ for $1\leq j,\, k\leq m$.
For instance, if $|t_j|<1$ for all $j$ we may integrate over the unit circle.

Let $f(x)$ denote the integrand in \eqref{wpi}. To explain the connection to elliptic hypergeometric series, we first use \eqref{egp} and Exercise \ref{epi} 
to show that 
$$\frac{f(q^k x)}{f(x)}=C^k q^{\binom k2(m-3)}x^{k(m-3)}\frac{\theta(q^{2k}x^2;p)}{\theta(x^2;p)}\prod_{j=0}^{m}\frac{(t_jx;q,p)_k}{(qx/t_j;q,p)_k}, $$
where  $C=(-1)^{m+1}q^{m-3}/t_0\dotsm t_{m}$. 
This clearly has a very-well-poised structure, but we dislike the quadratic exponent of $q$. To get rid of it,
we substitute  $t_j=p^{l_j}u_j$, where  $l_j$ are integers. Using again
Exercise \ref{epi}, we find that
$$\frac{f(q^k x)}{f(x)}=\left(Cp^{|l|}q^{-|l|}\right)^k \left(q^{\binom k2}x^{k}\right)^{m-3-2|l|}\frac{\theta(q^{2k}x^2;p)}{\theta(x^2;p)}\prod_{j=0}^{m}\frac{(u_jx;q,p)_k}{(qx/u_j;q,p)_k}. $$
Thus, it is natural to take $m$  odd and $|l|=(m-3)/2$. If we in addition assume that
\begin{equation}\label{ibc}t_0\dotsm t_m=(pq)^{(m-3)/2},\end{equation}
then $C=p^{-|l|}q^{|l|}$ and 
$$\frac{f(q^k x)}{f(x)}=\frac{\theta(q^{2k}x^2;p)}{\theta(x^2;p)}\prod_{j=0}^{m}\frac{(u_jx;q,p)_k}{(qx/u_j;q,p)_k}.$$
This gives in turn
\begin{align}\nonumber\frac{\Res_{x=u_0q^k}f(x)}{\Res_{x=u_0}f(x)}&=\frac{\lim_{x\rightarrow u_0q^k}(x-u_0q^k)f(x)}{\lim_{x\rightarrow u_0}(x-u_0)f(x)}=q^k\lim_{x\rightarrow u_0}\frac{f(q^kx)}{f(x)}\\
\label{frq}&=\frac{\theta(q^{2k}u_0^2;p)}{\theta(u_0^2;p)}\frac{(u_0^2,u_0u_1\dotsm,u_0u_m;q,p)_k}{(q,qu_0/u_1,\dots,qu_0/u_m;q,p)_k}\,q^k.
\end{align}
Thus, the (typically divergent) sum of residues at the points $x=u_0q^k$, $k\geq 0$, is a constant times
$${}_{m+5}V_{m+4}(u_0^2;u_0u_1,\dots,u_0u_m;q,p). $$ 
In contrast to \eqref{bi}, we are not claiming that the integral is equal to this sum. However, we can still think of \eqref{wpi} as a  substitute for the series ${}_{m+5}V_{m+4}$, at least when $m$ is odd and \eqref{ibc} holds.

\begin{problem}
In view of the definition of elliptic hypergeometric series, it is natural to call an integral
$$\int f(x)\,dx $$
\emph{elliptic hypergeometric} if  $f(qx)/f(x)$ is multiplicatively elliptic with period $p$. Show that the integral \eqref{wpi} is elliptic hypergeometric  if and only if
$t_0^2\dotsm  t_m^2=(pq)^{m-3}$. 
\end{problem}

\section{Spiridonov's elliptic beta integral}

In view of the Frenkel--Turaev formula, one may hope that $\eqref{wpi}$ can be computed in closed form for
$m=5$ and $t_0\dotsm t_5=pq$. Indeed, we have the following beautiful integral evaluation due to Spiridonov.
The limit case when $p\rightarrow 0$ (with $t_0,\dots,t_4$ fixed) is due to Nasrallah and Rahman \cite{nr}; the subsequent limit $t_4\rightarrow 0$ is the famous Askey--Wilson integral \cite{aw}.

 \begin{theorem}\label{sbp}
 Assume that $|p|,\,|q|<1$ and that $t_0,\dots,t_5$ are parameters such that 
 \begin{equation}\label{ftv}t_jt_k\notin p^{\mathbb Z_{\leq 0}}q^{\mathbb Z_{\leq 0}},\qquad 0\leq j,k\leq 5,\end{equation}
 \begin{equation}\label{sib}t_0\dotsm t_5=pq.\end{equation}
  Then,
\begin{equation}\label{sbi} I(t_0,\dots,t_5;p,q)
 =\frac{2\prod_{0\leq i<j\leq 5}\Gamma(t_it_j;p,q)}{(p;p)_\infty(q;q)_\infty}.\end{equation}
  \end{theorem}

We will give an elementary proof of Theorem \ref{sbp} found
 by Spiridonov \cite{sp2} a few years after the original proof in \cite{sp1}.
 Consider first $I(t_0,\dots,t_5;p,q)$ as a function of $p$. For fixed values of the parameters $t_j$, there are only finitely many values of $p$
 such that $|p|<1$ and \eqref{ftv} is violated. Outside these points, the integral is analytic in $p$. Thus, by analytic continuation, we may assume $0<p<1$. By symmetry, we may assume that $0<q<1$ and, again by analytic continuation, 
 \eqref{ndpq}.
  Consider now the integral as a function of $t_0$, where  $t_1,\dots,t_4$ are fixed and $t_5$ is  determined  from  \eqref{sib}. 
 It is analytic as long as $t_0$ avoids the points forbidden by \eqref{ftv}.
 Since $p$ and $q$ are real, these forbidden values are on a finite number of rays starting at the origin.
 We will assume that $t_0$ avoids these rays, that is,
 \begin{equation}\label{trc} t_jt_k\notin\mathbb R_{>0},\qquad j=0,\,5,\quad 0\leq k\leq 5.\end{equation}
 We proceed to  show that, under this condition, the quotient
 $$F(t_0)=\frac{I(t_0,\dots,t_5;p,q)}{\prod_{0\leq i<j\leq 5}\Gamma(t_it_j;p,q)} $$
 satisfies $F(qt_0)=F(t_0)$.

    Let 
$$f(t_0;x)=\frac{\prod_{j=0}^5\Gamma(t_jx^\pm;p,q)}{\prod_{0\leq i<j\leq 5}\Gamma(t_it_j;p,q)\cdot\Gamma(x^{\pm 2};p,q)}
$$
denote the integrand of $F(t_0)$ (apart from $1/2\pi\ti x$). We use \eqref{egf} to compute the difference 
\begin{multline*}f(t_0;x)-f(qt_0;x)=\frac{\Gamma(t_0x^\pm,\dots,t_4x^\pm,q^{-1}t_5x^\pm;p,q)}{\Gamma(x^{\pm 2},t_0t_5,qt_0t_1,\dots,qt_0t_4;p,q)\prod_{1\leq i<j\leq 5}\Gamma(t_it_j;p,q)}\\
\times\big\{\theta(t_0t_1,\dots,t_0t_4,q^{-1}t_5x^\pm;p)-\theta(q^{-1}t_1t_5,\dots,q^{-1}t_4t_5,t_0x^\pm;p)\big\}.
\end{multline*}
We now apply the case $n=2$ of  \eqref{sbt},  where we substitute
$$(x,a_1,a_2,b_1,\dots,b_4)\mapsto(\lambda t_0,\lambda x,\lambda/x,t_1/\lambda,\dots,t_4/\lambda),\qquad \lambda^2=\frac{pq}{t_0t_5}. $$
After simplification, we find that the  factor in brackets equals
$$\frac{t_0\theta(t_5/qt_0;p)}{x^{-1}\theta(x^2;p)}\big\{x^{-2}\theta(t_0x,\dots,t_4x,t_5x/q;p)-
x^2\theta(t_0/x,\dots,t_4/x,t_5/qx;p)\big\}.$$
 It follows that
\begin{equation}\label{fgi}f(t_0;x)-f(qt_0;x)=g(t_0;x)-g(t_0;qx), \end{equation}
where
$$g(t_0;x)=\frac{t_0x\theta(t_5/qt_0;p)\theta(x^2;q)\Gamma(t_0x,\dots,t_4x,t_5x/q,t_0q/x,\dots,t_4q/x,t_5/x;p,q)}{\Gamma(t_0t_5;p,q)\prod_{j=1}^4\Gamma(qt_0t_j;p,q)\prod_{1\leq i<j\leq 5}\Gamma(t_it_j;p,q)}. $$

Integrating \eqref{fgi} over a contour $\mathcal C$ gives
\begin{equation}\label{fgii}\int_{\mathcal C} f(t_0;x)\,\frac{dx}{2\pi\ti x}- \int_{\mathcal C} f(qt_0;x)\,\frac{dx}{2\pi\ti x}=\int_{\mathcal C} g(t_0;x)\,\frac{dx}{2\pi\ti x}- \int_{q\mathcal C} g(t_0;x)\,\frac{dx}{2\pi\ti x}.\end{equation}
We choose $\mathcal C$ so that the points \eqref{ip} (with $m=5$) are inside $\mathcal C$ and their reciprocals outside. Then, the first integral is equal to $F(t_0)$. 
The second integral equals $F(qt_0)$, provided that these conditions still hold when $t_0,\,t_5$ are
replaced by $t_0q,\,t_5/q$. This gives the additional requirement that the points
$x=p^jq^{-1}t_5 $ are inside $\mathcal C$ for $j\in\mathbb Z_{\geq 0}$ and their reciprocals outside $\mathcal C$.
We can choose $\mathcal C$ in this way, provided that none of the points $p^jq^{-1}t_5$ is equal to the reciprocal of one of the points
\eqref{ip}. This follows from our assumption \eqref{trc}.

The function $g(t_0;x)$ has poles at
$$x=t_lp^jq^{k+1},\quad x=t_5p^jq^k,\qquad 0\leq l \leq 4,\quad j,k\in\mathbb Z_{\geq 0} $$
and at
$$x=t_l^{-1}p^{-j}q^{-k},\quad x=t_5^{-1}p^{-j}q^{1-k},\qquad 0\leq l \leq 4,\quad j,k\in\mathbb Z_{\geq 0}. $$
The first set of poles are inside both the contours $q \mathcal C$ and $\mathcal C$, whereas the second set of poles are outside both contours. Thus, we can deform $\mathcal C$ to $q\mathcal C$ without crossing any poles of $g$. It follows that the right-hand side of \eqref{fgii} vanishes. This completes our proof that $F(qt_0)=F(t_0)$.  By symmetry, $F(pt_0)=F(t_0)$. Since $p$ and $q$ are real, we may iterate these equations without violating \eqref{trc}.
Thus, $F(p^kq^lt_0)=F(t_0)$ for $k,\,l\in\mathbb Z$. Since we assume \eqref{ndpq}, the points $p^kq^lt_0$ have a limit point in the open set defined by \eqref{trc}. 
By analytic continuation,  $F(t_0)=F$ is a constant.

To compute the constant $F$, we consider the limit $t_0\rightarrow t_1^{-1}$. The obstruction from letting $t_0=t_1^{-1}$ in the definition of $I$ comes from the condition that $x=t_0$ and $x=t_1$ are inside $\mathcal C$, whereas $x=t_0^{-1}$ and $x=t_1^{-1}$ are outside.
To resolve this problem, we write
\begin{equation}\label{fmc}F=\int_{\mathcal C}f(t_0;x)\,\frac{dx}{2\pi\ti x}=\Res_{x=t_0}\frac{f(t_0;x)}{x}-\Res_{x=t_0^{-1}}\frac{f(t_0;x)}{x}
+\int_{\mathcal C'}f(t_0;x)\,\frac{dx}{2\pi\ti x}, \end{equation}
where $\mathcal C'$ is a modification of $\mathcal C$ running outside $x=t_0$ and inside $x=t_0^{-1}$. As $f$ vanishes in the limit $t_0\rightarrow t_1^{-1}$, so does the integral over $\mathcal C'$. Moreover, since
 $f(t_0;x)=f(t_0;x^{-1})$, the first two terms can be combined and we obtain
$$F=2\lim_{t_0\rightarrow t_1^{-1}}\Res_{x=t_0}\frac{f(t_0;x)}{x}.$$
We compute
\begin{multline*}\Res_{x=t_0}\Gamma(t_0/x;p,q)\\
\begin{split}&=\lim_{x\rightarrow t_0}\frac{x-t_0}{(1-t_0/x)\prod_{j=0}^\infty(1-p^jt_0/x)(1-q^jt_0/x)}\prod_{j,k=0}^\infty\frac{1-p^{j+1}q^{k+1}x/t_0}{1-p^{j+1}q^{k+1}t_0/x}\\
&=\frac{t_0}{(p;p)_\infty(q;q)_\infty}\end{split}\end{multline*}
and consequently
$$ \Res_{x=t_0}\frac{f(t_0;x)}{x}=\frac{\prod_{j=1}^5\Gamma(t_j/t_0;p,q)}{(p;p)_\infty(q;q)_\infty\Gamma(t_0^{-2};p,q)\prod_{1\leq i<j\leq 5}\Gamma(t_it_j;p,q)}.$$
If $t_0\rightarrow t_1^{-1}$, this becomes
$$ \frac{1}{(p;p)_\infty(q;q)_\infty\prod_{2\leq i<j\leq 5}\Gamma(t_it_j;p,q)}=\frac 1{(p;p)_\infty(q;q)_\infty},$$
where we used \eqref{gi} and the fact that $t_2t_3t_4t_5=pq$ in the limit.
 In conclusion, $F= 2/{(p;p)_\infty(q;q)_\infty}$,
which is Spiridonov's integral evaluation.

\begin{problem}
By considering the integral
$$\int\frac{\Gamma(cz^\pm w^\pm;p,q)\prod_{j=1}^4\Gamma(a_jz^\pm,b_jw^\pm;p,q)}{\Gamma(z^{\pm2},w^{\pm 2};p,q)}\frac{dz}{2\pi\ti z}\frac{dw}{2\pi\ti w}, $$
prove that
$$I(t_1,\dots,t_8;p,q)=\prod_{1\leq j<k\leq 4}\Gamma(t_jt_k,t_{j+4}t_{k+4};p,q)I(t_1/\lambda,\dots,t_4/\lambda,t_5\lambda,\dots,t_8\lambda), $$
where $\lambda^2=t_1t_2t_3t_4/pq=pq/t_5t_6t_7t_8$.
Iterating this identity, obtain several further integral transformations (see \cite[\S 5.1]{ss}).
\end{problem}

\begin{problem}
Consider the limit $t_0t_1\rightarrow q^{-n}$ of \eqref{sbi}. Generalizing the splitting \eqref{fmc} and using \eqref{frq}, deduce the Frenkel--Turaev summation.
\end{problem}

\begin{problem}\label{awe}
Show, under appropriate assumptions on the parameters and the contour of integration, the one-parameter family of biorthogonality relations
\begin{multline*} \int\frac{\Gamma(t_0x^\pm,\dots, t_5x^\pm;p,q)\theta(\lambda x^\pm;q)}{\Gamma(x^{\pm 2};p,q)}\\
\times r_k(X(x);t_0,t_1,t_2,t_3,t_4,t_5;q,p)\,r_l(X(x);t_0,t_1,t_2,t_3,t_5,t_4;q,p)\,\frac{dx}{2\pi\ti x}
=0,
\end{multline*}
where $k\neq l$, $t_0\dotsm t_5=q$, the rational functions $r_k$ are defined in \eqref{wof} and $\lambda$ is a free parameter.
This gives an elliptic analogue of Rahman's biorthogonal rational functions \cite{ra}, which generalize
 the Askey--Wilson polynomials \cite{aw}.\footnote{Hint: Use the symmetry of Exercise \ref{rse} to expand $r_k$ as a sum with numerator parameters $t_0x^\pm$ and $r_l$ as a sum with  numerator parameters
$t_1x^\pm$. Then use that, since 
 $\theta(\lambda x^\pm;q)$ is in the linear span of $\theta(t_2x^\pm;q)$ and $\theta(t_3x^\pm;q)$, it is sufficient to take $\lambda=t_3$.
One can give more general  two-index biorthogonality relations 
 for functions of the form $r_{k_1}(\dotsm;p,q)r_{k_2}(\dotsm;q,p)$, see \cite{scb}.}
\end{problem}

\chapter{Solvable lattice models}\label{slc}

\section{Solid-on-solid  models}

We will now explain how elliptic hypergeometric series first appeared, as fused Boltzmann weights for Baxter's elliptic solid-on-solid model.
The main reference for Chapter \ref{slc} is \cite{d}. 
The reader who wants more background on  exactly solvable  models in statistical mechanics is referred to the first few chapters of \cite{jm}
for a brief introduction and  to the standard textbook \cite{ba} for more details.

The goal of statistical mechanics is to predict the large-scale behaviour of a system described by local rules. 
In solid-on-solid (SOS) models, the system can be
viewed as a random surface. Let us first consider a model whose \emph{states} are rectangular arrays of  fixed size with real entries. We may think of the entries as  
the height of a discrete surface over the rectangle. 
To get a statistical model, we need to associate a weight to each state, which is proportional to the probability that  the state is assumed.\footnote{Physically, the
weight is $e^{-E/kT}$, where $k>0$ is Boltzmann's constant, $E$ the energy of the state and $T$ the temperature. Thus, high energy states are less likely than low energy states, but become more
likely as temperature increases.}  
The weight is the product of local \emph{Boltzmann weights},
associated to $2\times 2$-blocks of adjacent entries in the array.
If our array is $(M+1)\times (N+1)$, the $2\times 2$-blocks naturally  form an $(M\times N)$ array. Giving the blocks coordinates $(i,j)$ in a standard way,
suppose that the block with coordinates $(i,j)$ is $\left[\begin{smallmatrix}a&b\\c&d\end{smallmatrix}\right]$. 
We then assign to this block the  Boltzmann weight 
$$W\bigg(\begin{matrix}a&b\\c&d\end{matrix}\,\bigg|\,u_i,v_j\bigg),$$
 where we for now think of $W$ as an arbitrary function
 of six variables, 
   and where $u_1,\dots,u_M,v_1,\dots,v_N$ are parameters associated to the vertical and horizontal lines separating the heights.
   These parameters are known as \emph{rapidities} (or spectral parameters) and the lines are called \emph{rapidity lines}. 
As an example, the state
$$\begin{matrix}1&2&3\\4&5&6\\7&8&9\end{matrix} $$
has weight
\begin{align*}&W\left(\begin{matrix}1&2\\4&5\end{matrix}\,\bigg|\,u_1,v_1\right)\cdot W\left(\begin{matrix}2&3\\5&6\end{matrix}\,\bigg|\,u_1,v_2\right)\\
\times\, &W\left(\begin{matrix}4&5\\7&8\end{matrix}\,\bigg|\,u_2,v_1\right)\cdot W\left(\begin{matrix}5&6\\8&9\end{matrix}\,\bigg|\,u_2,v_2\right). \end{align*}

In the models that we will consider, only finitely many states are allowed. For instance, for Baxter's
elliptic SOS model described in \S \ref{ess},
one rule is that the height of adjacent squares differ by exactly one.\footnote{This condition is natural for  body-centered cubic crystals such as iron at room temperature.}
Since we  may shift all heights by an arbitrary real number we need further restrictions, which we will
take to be boundary conditions. For instance,
we may fix
the whole boundary or just  the height at a corner.
 Assuming in addition that $W$ is positive, we can then 
define the probability of a state to be its weight divided by the \emph{partition function}
$$\sum_{\text{states}} \text{weight},$$
where we sum over all states satisfying our boundary conditions.

We now observe that  SOS models make sense on more general geometries.  The rapidity lines, separating regions with constant height,  could be quite
arbitrary oriented  curves in a portion of the plane. 
The main restriction is that only two curves may cross at any point.  Then, each crossing looks like
\setlength{\unitlength}{2pt}
$$ \begin{picture}(35,35)
\put(13,23){$a$}
\put(23,23){$b$}
\put(23,13){$d$}
\put(13,13){$c$}
\put(18.5,32){$v$}
\put(32,18.5){$u$}
\put(20,10){\vector(0,1){20}}
\put(10,20){\vector(1,0){20}}
\end{picture}$$
to which we assign the  weight 
$$W\left(\begin{matrix}a&b\\c&d\end{matrix}\,\bigg|\,u,v\right).$$
Note that the orientation determines how the heights and rapidities should be inserted in $W$.
 To give an example, the state
 \setlength{\unitlength}{5pt}
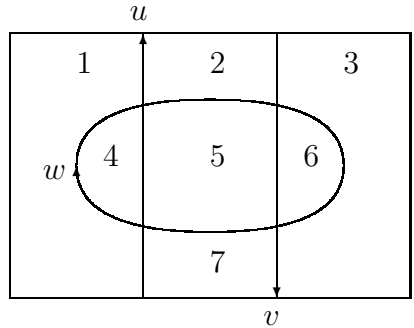
$$\begin{picture}(40,24)(0,-2)
\put(10,0){\line(1,0){30}}
\put(10,0){\line(0,1){20}}
\put(10,20){\line(1,0){30}}
\put(40,0){\line(0,1){20}}
\put(20,0){\vector(0,1){20}}
\put(30,20){\vector(0,-1){20}}
\qbezier(15,10)(15,15)(25,15)
\qbezier(25,15)(35,15)(35,10)
\qbezier(35,10)(35,5)(25,5)
\qbezier(25,5)(15,5)(15,10)
\put(15,17){$1$}
\put(25,17){$2$}
\put(35,17){$3$}
\put(17,10){$4$}
\put(25,10){$5$}
\put(32,10){$6$}
\put(25,2){$7$}
\put(19,21){$u$}
\put(29,-2){$v$}
\put(15,10){\vector(0,1){0}}
\put(12.5,9){$w$}
\end{picture}
$$
 has weight
$$W\left(\begin{matrix}1&2\\4&5\end{matrix}\,\bigg|\,w,u\right)\cdot W\left(\begin{matrix}3&6\\2&5\end{matrix}\,\bigg|\,v,w\right)\cdot
W\left(\begin{matrix}1&4\\7&5\end{matrix}\,\bigg|\,u,w\right)\cdot W\left(\begin{matrix}3&7\\6&5\end{matrix}\,\bigg|\,w,v\right). $$

 \section{The Yang--Baxter equation}
 
 In the models of interest to us, the 
 weights satisfy the \emph{Yang--Baxter equation}, which can be viewed as
 an integrability criterion.\footnote{To be slightly more precise, Baxter's elliptic SOS model is closely related to a one-dimensional quantum mechanical model known as the XYZ spin chain. Using the Yang--Baxter equation, one can find an infinite family of operators commuting with the Hamiltonian of the spin chain. This is a quantum analogue of Liouville integrability, where there exists a maximal set of invariants that Poisson commute with the classical Hamiltonian.} 
Roughly speaking, the Yang--Baxter equation gives a natural way to make sense of triple crossings. 
More precisely, we want to allow the type of crossing to the
 left in the following picture, but not the one to the right. 
$$ 
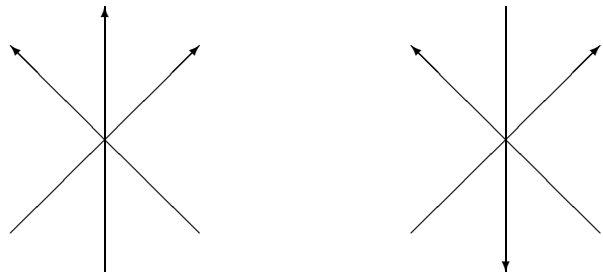
\begin{picture}(60,20)
\put(2.93,2.93){\vector(1,1){14.14}}
\put(10,0){\vector(0,1){20}}
\put(17.07,2.93){\vector(-1,1){14.14}}
\put(32.93,2.93){\vector(1,1){14.14}}
\put(40,20){\vector(0,-1){20}}
\put(47.07,2.93){\vector(-1,1){14.14}}
\end{picture}
$$

Imagine that our triple crossing is composed of three single crossings viewed from a distance. This can happen in two ways, as
illustrated in the following picture. We have also introduced symbols for the adjacent heights and the rapidities. 
$$ 
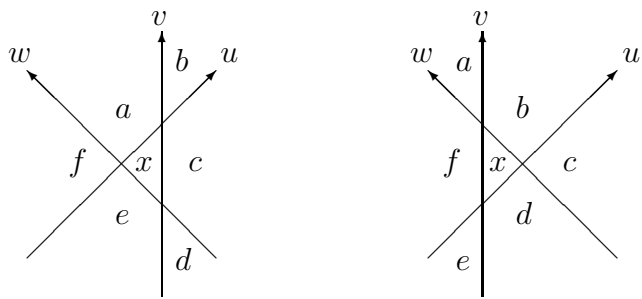
\begin{picture}(60,22)
\put(32.93,2.93){\vector(1,1){14.14}}
\put(37,0){\vector(0,1){20}}
\put(47.07,2.93){\vector(-1,1){14.14}}
\put(2.93,2.93){\vector(1,1){14.14}}
\put(13,0){\vector(0,1){20}}
\put(17.07,2.93){\vector(-1,1){14.14}}
\put(35,17){$a$}
\put(39.5,13.5){$b$}
\put(43,9.5){$c$}
\put(39.5,5.5){$d$}
\put(35,2){$e$}
\put(34,9.5){$f$}
\put(37.5,9.5){$x$}
\put(14,17){$b$}
\put(9.5,13.5){$a$}
\put(6,9.5){$f$}
\put(9.5,5.5){$e$}
\put(14,2){$d$}
\put(15,9.5){$c$}
\put(11,9.5){$x$}
\put(31.6,17.7){$w$}
\put(36.3,20.5){$v$}
\put(47.5,17.7){$u$}
\put(17.4,17.7){$u$}
\put(12.2,20.5){$v$}
\put(1.5,17.7){$w$}
\end{picture}$$
We now postulate that the corresponding two systems have
the same partition function. We allow $a,\dots,f$ to be fixed, for instance by boundary conditions, but  sum over all possibilities for $x$ (as before, we assume that the resulting sums are finite). This gives the Yang--Baxter equation in the form
\begin{multline}\label{yb}\sum_x W\bigg(\begin{matrix}a&b\\x&c\end{matrix}\,\bigg|\,u,v\bigg)
W\bigg(\begin{matrix}f&a\\e&x\end{matrix}\,\bigg|\,u,w\bigg)
W\bigg(\begin{matrix}e&x\\d&c\end{matrix}\,\bigg|\,v,w\bigg)\\
=\sum_xW\bigg(\begin{matrix}f&a\\x&b\end{matrix}\,\bigg|\,v,w\bigg)
W\bigg(\begin{matrix}x&b\\d&c\end{matrix}\,\bigg|\,u,w\bigg)
W\bigg(\begin{matrix}f&x\\e&d\end{matrix}\,\bigg|\,u,v\bigg).
\end{multline}

There is another natural relation for Boltzmann weights known as the unitarity relation.
Pictorially, it means that two consecutive crossings of rapidity lines cancel, that is,
that the systems
 \setlength{\unitlength}{4pt}
\begin{equation}\label{urp}\begin{picture}(60,22)
\qbezier(0,0)(20,10)(0,20)
\qbezier(15,0)(-5,10)(15,20)
\put(7,18){$a$}
\put(7,9.5){$x$}
\put(12,9.5){$b$}
\put(7,1){$c$}
\put(2,9.5){$d$}
\put(0,20){\vector(-2,1){0}}
\put(15,20){\vector(2,1){0}}
\put(-2,19){$u$}
\put(16,19){$v$}
\put(40,0){\vector(0,1){20}}
\put(50,0){\vector(0,1){20}}
\put(39.2,21){$u$}
\put(49.2,21){$v$}
\put(45,18){$a$}
\put(45,1){$c$}
\put(37,9.5){$d$}
\put(52,9.5){$b$}
\end{picture}
  \end{equation}
  have the same partition function. Note that if $a\neq c$ the system on the right does not exist (heights may only change as we cross a rapidity line) so we set its weight to zero. If $a=c$, there are no crossings, so the weight is the empty product $1$.  On the left, we should as before sum over $x$. This leads to the condition
  \begin{equation}\label{ur}\sum_x W\bigg(\begin{matrix}d&x\\c&b\end{matrix}\,\bigg|\,u,v\bigg)W\bigg(\begin{matrix}d&a\\x&b\end{matrix}\,\bigg|\,v,u\bigg)=\delta_{ac}.
  \end{equation}


\section{The $R$-operator}

Focusing on the relative changes in height, 
it is often useful to rewrite the Boltzmann weights in the notation
$$W\bigg(\begin{matrix}a&b\\c&d\end{matrix}\,\bigg|\,u,v\bigg)=R^{c-a,d-c}_{d-b,b-a}(a|u,v),$$
or equivalently
 $$R^{mn}_{kl}(\lambda|u,v)=W\bigg(\begin{matrix}\lambda &\lambda+l\\\lambda+m&\lambda+N\end{matrix}\,\bigg|\,u,v\bigg),\qquad k+l=m+n=N.$$

We want to view $R^{mn}_{kl}$ as matrix elements of an operator, which we can think of as representing a crossing of two rapidity lines. To this end, let $\Lambda$ be the set of allowed 
differences between adjacent heights, $V$ a vector space with basis $(e_\lambda)_{\lambda\in \Lambda}$
and $R(\lambda|u,v)$ the operator on $V\otimes V$ defined by
$$R(\lambda|u,v)(e_k\otimes e_l)=\sum_{m+n=k+l}R^{kl}_{mn}(\lambda|u,v)\,(e_m\otimes e_n).$$
In \eqref{yb}, let 
$$ i=e-f, \quad j=d-e,\quad k=c-d,\quad l=c-b,\quad m=b-a,\quad n=a-f$$
denote the various height differences encountered when travelling around the triple crossing. 
We replace $x$ by $c-x$ on the left-hand side and $x+f$ on the right  and finally let $f=\lambda$. This gives
\begin{multline}\label{cyb}\sum_{x} R^{l+m-x,x}_{lm}(\lambda+n|u,v)R^{i,j+k-x}_{l+m-x,n}(\lambda|u,w) R^{jk}_{x,j+k-x}(\lambda+i|v,w)\\
=\sum_{x} R^{x,m+n-x}_{mn}(\lambda|v,w)R^{i+j-x,k}_{l,m+n-x}(\lambda+x|u,w)R^{ij}_{i+j-x,x}(\lambda|u,v).
\end{multline}
We consider the two sides as matrix elements for operators on $V\otimes V\otimes V$,
where we act on $e_i\otimes e_j\otimes e_k$ and pick out the coefficient of $e_l\otimes e_m\otimes e_n$.
The resulting  coordinate-free form of \eqref{yb} was first given by Felder \cite{f}. It can be written
\begin{multline*}R^{12}(\lambda+h^3|u,v)R^{13}(\lambda|u,w)R^{23}(\lambda+h^1|v,w)\\
=R^{23}(\lambda|v,w)R^{13}(\lambda+h^2|u,w)R^{12}(\lambda|u,v). \end{multline*}
Here, the upper indices determine the spaces in the tensor product where the operators are acting, and $h$ is the grading operator $he_j=je_j$. 
For instance,
\begin{align*}R^{12}(\lambda|u,v)(e_i\otimes e_j\otimes e_k)&=(R(\lambda|u,v) (e_i\otimes e_j))\otimes e_k\\
&=\sum_x R^{ij}_{i+j-x,x}(\lambda|u,v)\,e_{i+j-x}\otimes e_{x}\otimes e_k,\\
 R^{23}(\lambda+h^1|v,w)(e_i\otimes e_j\otimes e_k)&=e_i\otimes(R(\lambda+i|v,w)(e_j\otimes e_k))\\
&=\sum_x R^{jk}_{x,j+k-x}(\lambda+i|v,w)\,e_i\otimes e_{x}\otimes e_{j+k-x}.\end{align*}

We mention that there are several versions of the Yang--Baxter equation. The one encountered here is sometimes called the
 quantum dynamical Yang--Baxter equation.
Important special cases include the quantum Yang--Baxter equation, when $R(\lambda|u,v)$ is independent of $\lambda$,
and the hexagon identity for $6j$-symbols (Racah and $q$-Racah polynomials), when $R(\lambda|u,v)$ is independent of
$u$ and $v$.

\begin{problem}\label{up}
Show that the unitarity relation \eqref{ur} can be expressed as
$$R^{12}(\lambda|u,v) R^{21}(\lambda|v,u)=\operatorname{Id}. $$
\end{problem}

\section{The elliptic SOS model}\label{ess}

The elliptic SOS model (also called eight-vertex-solid-on-solid model) was introduced by Baxter \cite{bx} in his solution of a related model known as the eight-vertex model. In the elliptic SOS model, neighbouring heights differ by exactly $1$. In particular, if one height is $a$ then necessarily all heights are in $a+\mathbb Z$. 
The Boltzmann weights only depend on the quotient of the two rapidities at a crossing. For this reason, we write
\begin{equation}\label{rq} W\bigg(\begin{matrix}a&b\\c&d\end{matrix}\,\bigg|\,u,v\bigg) =W\bigg(\begin{matrix}a&b\\c&d\end{matrix}\,\bigg|\,u/v\bigg),\qquad
R_{kl}^{mn}(\lambda|u,v)=R_{kl}^{mn}(\lambda|u/v).\end{equation}

The indices in \eqref{rq} satisfy
$k,\,l,\,m,\, n\in\{\pm 1\}$ and $k+l=m+n$, which has six solutions. Writing $\pm$ instead of $\pm 1$, the Boltzmann weights are given by
\begin{subequations}\label{esw}
\begin{align}R^{++}_{++}(\lambda|u)&=R^{--}_{--}(\lambda|u)=1, \\
R^{+-}_{+-}(\lambda|u)&=\frac{\theta(q^{1-\lambda},u;p)}{\theta(q^{-\lambda},uq;p)}, & R^{-+}_{-+}(\lambda|u)&=\frac{\theta(q^{\lambda+1},u;p)}{\theta(q^\lambda,uq;p)},\\
R^{+-}_{-+}(\lambda|u)&=\frac{\theta(q,q^{-\lambda} u;p)}{\theta(q^{-\lambda},uq;p)}, & R^{-+}_{+-}(\lambda|u)&=\frac{\theta(q,q^\lambda u;p)}{\theta(q^\lambda,uq;p)}.
\end{align}
\end{subequations}
Here, $p$ and $q$ are fixed parameters with $|p|<1$. We will assume \eqref{ndpq}, though the case when $q$ is a root of unity is in fact of special interest. 
To make physical sense, one should choose the parameters so that the Boltzmann weights are positive, but that will not be a concern for us.
It will be useful to note the symmetries
\begin{equation}\label{rsy}
R_{kl}^{mn}(\lambda|u)=R_{-k,-l}^{-m,-n}(-\lambda|u)=R^{nm}_{lk}(-\lambda-k-l|u)
\end{equation}
or, equivalently,
\begin{equation}W\bigg(\begin{matrix}a&b\\c&d\end{matrix}\,\bigg|\,u\bigg)=W\bigg(\begin{matrix}-a&-b\\-c&-d\end{matrix}\,\bigg|\,u\bigg)
\label{bws}=W\bigg(\begin{matrix}d&b\\c&a\end{matrix}\,\bigg|\,u\bigg). \end{equation}

\begin{theorem}\label{bdt}
The Boltzmann weights \eqref{esw} satisfy the  Yang--Baxter equation \eqref{cyb}.
\end{theorem}

Unfortunately, we cannot present an elegant proof of
Theorem~\ref{bdt}, so we resort to brute force verification. A priori, we need to check \eqref{cyb} for 
each choice of $i,j,k,l,m,n\in\{\pm 1\}$ with $i+j+k=l+m+n$. This leads to $20$ equations. Fortunately, the number can be reduced by exploiting symmetries of the Boltzmann weights.
Indeed, applying the first equation in \eqref{rsy}  to
\eqref{cyb} and replacing $\lambda$ by $-\lambda$ gives back \eqref{cyb} with 
$(i,j,k,l,m,n)\mapsto(-i,-j,-k,-l,-m,-n).$
 Thus, we may assume
$i+j+k>0$, leaving us with $10$ equations. Similarly, applying the second equation in \eqref{rsy} to \eqref{cyb} and replacing $(\lambda,u,v,w)\mapsto(-\lambda-i-j-k,w^{-1},v^{-1},u^{-1})$ leads to \eqref{cyb} after the permutations $i\leftrightarrow k$ and $l\leftrightarrow n$.
Thus, we may assume $i\geq k$ and if $i=k$ we may in addition assume $l\geq n$. Now, we are down to the six equations given in the following table, where we also write $x_L$ and $x_R$ for the admissible values
of $x$ at the left-hand and right-hand side of \eqref{cyb}.
\begin{center}\begin{tabular}{cccccccc}
$i$&$j$&$k$&$l$&$m$&$n$&$x_L$&$x_R$\\
\hline
+&+&+&+&+&+&+&+\\
+&+&-&+&+&-&+&+\\
+&+&-&+&-&+&$\pm$&+\\
+&+&-&-&+&+&$\pm$&+\\
+&-&+&+&+&-&+&$\pm$\\
+&-&+&+&-&+&$\pm$&$\pm$
\end{tabular}\end{center}
In particular, the number of terms in these identities are, respectively, $2$, $2$, $3$, $3$, $3$ and $4$.
The two-term identities are trivial. It is easy to check that the  three-term identities are all equivalent to
Weierstrass's identity \eqref{ttr}. Finally, the four-term 
identity has the form
\begin{align*}
&R^{-+}_{+-}(\lambda+1|u/v)R^{+-}_{-+}(\lambda|u/w)R^{-+}_{+-}(\lambda+1|v/w)\\
&\quad+R^{+-}_{+-}(\lambda+1|u/v)R^{++}_{++}(\lambda|u/w)R^{-+}_{-+}(\lambda+1|v/w)\\
&=R^{+-}_{-+}(\lambda|v/w)R^{-+}_{+-}(\lambda+1|u/w)R^{+-}_{-+}(\lambda|u/v)\\
&\quad+R^{-+}_{-+}(\lambda|v/w)R^{++}_{++}(\lambda-1|u/w)R^{+-}_{+-}(\lambda|u/v)
\end{align*}
or, equivalently,
\begin{multline}\label{ftt}
\frac{\theta(q;p)^3\theta(q^{\lambda+1}u/v,q^{-\lambda}u/w,q^{\lambda+1}v/w;p)}{\theta(q^{\lambda+1},uq/v,q^{-\lambda},uq/w,q^{\lambda+1},vq/w;p)}+\frac{\theta(q^{-\lambda},u/v,q^{\lambda+2},v/w;p)}{\theta(q^{-\lambda-1},uq/v,q^{\lambda+1},vq/w;p)}\\
=\frac{\theta(q;p)^3\theta(q^{-\lambda}v/w,q^{\lambda+1}u/w,q^{-\lambda}u/v;p)}{\theta(q^{-\lambda},vq/w,q^{\lambda+1},uq/w,q^{-\lambda},uq/v;p)}
+\frac{\theta(q^{\lambda+1},v/w,q^{1-\lambda},u/v;p)}{\theta(q^{\lambda},vq/w,q^{-\lambda},uq/v;p)}.
\end{multline}
We  prove this by  interpolation.
We first consider both sides as functions of $v$.
By Proposition \ref{elp}, it is enough to verify the identity for two values, $v=v_1$ and $v=v_2$, such that neither $v_1/v_2$ nor
$v_1v_2/uw$ is in $p^{\mathbb Z}$. We choose $v_1=u$, which gives a trivial identity, and $v_2=q^\lambda w$, which after simplification gives
$$-q^{\lambda-1}\,\theta(q;p)^3\theta(q^{2\lambda+1};p)+\theta(q^\lambda;p)^3\theta(q^{\lambda+2};p)= \theta( q^{\lambda+1};p)^3\theta( q^{\lambda-1};p).$$
This is a special case of Weierstrass's identity \eqref{ttr}.

\begin{problem}
Give an alternative proof of \eqref{ftt} by rearranging the terms and using \eqref{ttr}. 
\end{problem}

\begin{problem}\label{ue}
Show directly that the unitarity relation   \eqref{ur} holds in the elliptic SOS model.
(We will see another way to do this in Exercise~\ref{usp}.)
\end{problem}

\section{Fusion and elliptic hypergeometry}
\label{fhs}

For Baxter's model, adjacent heights differ by exactly $1$. One can  obtain less restrictive models by applying a so called fusion procedure. Date et al.\ \cite{d} found that the Boltzmann weights of these fused models are given by
${}_{12}V_{11}$-sums. We will  give a new approach to this result, by relating it to our construction of ${}_{12}V_{11}$-sums as connection coefficients described in \S \ref{tves}.
  
  As a starting point, the following  result gives a description of Baxter's Boltzmann weights as connection coefficients in a two-dimensional space of theta functions.

\begin{proposition}\label{gbp}
For $c=a\pm 1$, let $\phi(a,c|u)$ denote the function
$$\phi(a,c|u)(x)=\frac{\theta(q^{a(c-a)/2}\sqrt u x^\pm;p)}{ q^{a(c-a)/2}\sqrt u}. $$
Then, the Boltzmann weights of the elliptic SOS model are determined by the expansion
\begin{equation}\label{gbe}\phi(b,d|u)=\sum_{c} W\bigg(\begin{matrix}a&b\\c&d\end{matrix}\,\bigg|\,u\bigg)\phi(a,c|uq).\end{equation}
\end{proposition}

\begin{proof}
Fixing $a$, there are four possible choices for the pair $(b,d)$, corresponding to the identities
\begin{align*} \phi(a+1,a|u)&=\sum_{j=\pm 1}W\bigg(\begin{matrix}a&a+1\\a+j&a\end{matrix}\,\bigg|\,u\bigg)\phi(a,a+j|uq),\\
\phi(a-1,a|u)&=\sum_{j=\pm 1}W\bigg(\begin{matrix}a&a-1\\a+j&a\end{matrix}\,\bigg|\,u\bigg)\phi(a,a+j|uq),\\
\phi(a+1,a+2|u)&=W\bigg(\begin{matrix}a&a+1\\a+1&a+2\end{matrix}\,\bigg|\,u\bigg)\phi(a,a+1|uq),\\
\phi(a-1,a-2|u)&=W\bigg(\begin{matrix}a&a-1\\a-1&a-2\end{matrix}\,\bigg|\,u\bigg)\phi(a,a-1|uq)
\end{align*}
or, equivalently,
\begin{align*}
\theta(\sqrt u q^{-\frac{\lambda+1}{2}}x^\pm;p)&=q^{-\lambda-1}R^{+-}_{-+}(\lambda|u)\theta(\sqrt u q^{\frac{\lambda+1}{2}}x^\pm;p)\\
&\quad+q^{-1}R^{-+}_{-+}(\lambda|u)\theta(\sqrt u q^{\frac{1-\lambda}{2}}x^\pm;p),\\
\theta(\sqrt u q^{\frac{\lambda-1}{2}}x^\pm;p)&=q^{-1}R^{+-}_{+-}(\lambda|u)\theta(\sqrt u q^{\frac{\lambda+1}{2}}x^\pm;p)\\
&\quad+q^{\lambda-1}R^{-+}_{+-}(\lambda|u)\theta(\sqrt u q^{\frac{1-\lambda}{2}}x^\pm;p),\\
\theta(\sqrt u q^{\frac{\lambda+1}{2}}x^\pm;p)&=R^{++}_{++}(\lambda|u)\theta(\sqrt u q^{\frac{\lambda+1}{2}}x^\pm;p),\\
\theta(\sqrt u q^{\frac{1-\lambda}{2}}x^\pm;p)&=R^{--}_{--}(\lambda|u)\theta(\sqrt u q^{\frac{1-\lambda}{2}}x^\pm;p).
\end{align*}
 The first two expansions are of the form \eqref{we}, so the coefficients are given by Weierstrass's identity \eqref{ttr}.
 The second two identities are trivial. This leads to the explicit expressions for Boltzmann weights given in  \eqref{esw}. 
\end{proof}

Fusion of the SOS model corresponds to iterating \eqref{gbe}. We  have
\begin{align*}\phi(b,d|u)&=\sum_f W\bigg(\begin{matrix}e&b\\f&d\end{matrix}\,\bigg|\,u\bigg)\phi(e,f|uq)\\
&=\sum_{fc} W\bigg(\begin{matrix}a&e\\c&f\end{matrix}\,\bigg|\,uq\bigg)W\bigg(\begin{matrix}e&b\\f&d\end{matrix}\,\bigg|\,u\bigg)  \phi(a,c|uq^2).
  \end{align*}
As the two functions $\phi(a,a\pm 1|uq^2)$ are generically linearly independent, it follows that
\begin{equation}\label{efe} \sum_{f} W\bigg(\begin{matrix}a&e\\c&f\end{matrix}\,\bigg|\,uq\bigg)W\bigg(\begin{matrix}e&b\\f&d\end{matrix}\,\bigg|\,u\bigg) 
\end{equation}
is independent of $e$, as long as $|a-e|=|b-e|=1$. Equivalently, the partition function for
the system
$$
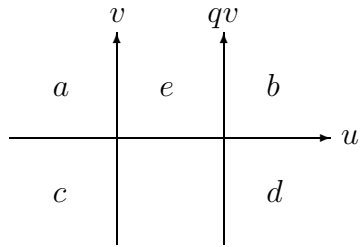
\begin{picture}(35,25)
\put(0,10){\vector(1,0){30}}
\put(10,0){\vector(0,1){20}}
\put(20,0){\vector(0,1){20}}
\put(31,9.5){$u$}
\put(9.3,21){$v$}
\put(18.5,21){$qv$}
\put(4,14){$a$}
\put(24,14){$b$}
\put(4,4){$c$}
\put(14,14){$e$}
\put(24,4){$d$}
 \end{picture} $$
where we sum over the admissible heights of the empty slot, is independent of $e$. Thus, we can forget about $e$, and
think of the vertical lines as coalescing. We may
view this sum as a Boltzmann weight for a \emph{fused} SOS-model, for which $a-c$ and $b-d$ are in $\{\pm 1\}$, whereas $b-a$ and $d-c$ are in $\{-2,0,2\}$.

 Next, we observe that fusion works also in the vertical direction. Indeed,  
 by  \eqref{bws}, the partition function for
$$
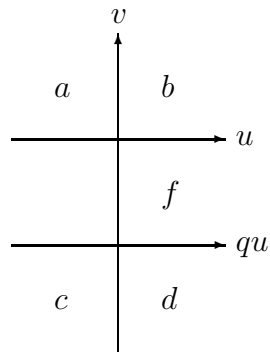
\begin{picture}(25,35)
\put(0,10){\vector(1,0){20}}
\put(0,20){\vector(1,0){20}}
\put(10,0){\vector(0,1){30}}
\put(21,19.5){$u$}
\put(21,9.5){$qu$}
\put(9.3,31){$v$}
\put(4,24){$a$}
\put(14,24){$b$}
\put(4,4){$c$}
\put(14,14){$f$}
\put(14,4){$d$}
 \end{picture} $$
is independent of $f$.  Iterating fusion in both directions, we find that the partition function for
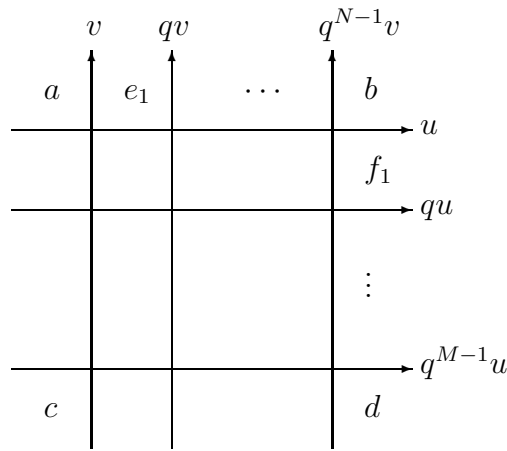
\setlength{\unitlength}{3pt}
\begin{equation}\label{if}\begin{picture}(55,55)
\put(0,10){\vector(1,0){50}}
\put(0,30){\vector(1,0){50}}
\put(0,40){\vector(1,0){50}}
\put(10,0){\vector(0,1){50}}
\put(20,0){\vector(0,1){50}}
\put(40,0){\vector(0,1){50}}
\put(51,39.5){$u$}
\put(51,29.5){$qu$}
\put(51,9.5){$q^{M-1}u$}
\put(9.3,52){$v$}
\put(18.3,52){$qv$}
\put(38.3,52){$q^{N-1}v$}
\put(4,44){$a$}
\put(44,44){$b$}
\put(4,4){$c$}
\put(44,34){$f_1$}
\put(44,4){$d$}
\put(14,44){$e_1$}
\put(29,44){$\dotsm$}
\put(44,19){$\vdots$}
 \end{picture} \end{equation}
where we fix the corners together with the top and right boundary and sum over the admissible heights of all other squares, 
is independent of the interior boundary heights $e_1,\dots,e_{N-1}$, $f_1,\dots,f_{M-1}$, provided that  adjacent boundary heights differ by $1$. We denote the resulting quantity 
$$W_{MN}\bigg(\begin{matrix}a&b\\c&d\end{matrix}\,\bigg|\,u/v\bigg). $$
Here,
$a-c$ and $b-d$ are in $\{-M,2-M,\dots,M\}$ whereas $a-b$ and $c-d$ are in  $\{-N,2-N,\dots,N\}$. 
We view $W_ {MN}$ as Boltzmann weights for a model, where each rapidity line is labelled by a multiplicity
that gives the maximal allowed height difference across that line.

In order to identify the  fused Boltzmann weights $W_{MN}$ with elliptic hypergeometric sums, we will need the following fact.

\begin{lemma}\label{pil} Assuming $|b-a|=|c-b|=1$,
the product $\phi(a,b|uq^{-1})\phi(b,c|u)$ is independent of $b$.
\end{lemma}

\begin{proof}
If $|a-c|=2$ there is only one admissible value for $b$ and there is nothing to prove. Else, $a=c$ and the claim is that
$$ \phi(a,a+1|uq^{-1})\phi(a+1,a|u)=\phi(a,a-1|uq^{-1})\phi(a-1,a|u).$$
This is trivial to verify.
\end{proof}

We will write
$$\phi_M(a,c|u) =\phi(a,b_1|uq^{1-M})\phi(b_1,b_2|uq^{2-M})\dotsm\phi(b_{M-1},c|u), $$
where $a-c\in\{-M,2-M,\dots,M\}$ and 
\begin{equation}\label{sd}b_1-a,\ b_2-b_1,\dots,\ c-b_{M-1}\in\{\pm 1\}.\end{equation} 
  By Lemma \ref{pil}, $\phi_M$ is independent of the parameters $b_j$. Thus, we may take the first $(M+c-a)/2$ of the numbers \eqref{sd} as 
  $1$ and the remaining $(M+a-c)/2$ as $-1$. After simplification, this gives the explicit formula
  \begin{multline*}\phi_M(a,c|u)=\frac{q^{\frac 14(c^2-a^2-M^2)}}{(\sqrt u)^M}\\
  \times(q^{(1-M+a)/2}\sqrt u x^\pm;q,p)_{(M+c-a)/2} (q^{(1-M-a)/2}\sqrt u x^\pm;q,p)_{(M+a-c)/2}. 
  \end{multline*}

We can now generalize Proposition \ref{gbp} to fused models.

\begin{theorem}\label{fbt}
The  Boltzmann weights of the fused elliptic SOS models satisfy
\begin{equation}\label{gbm}\phi_M(b,d|u)=\sum_{c} W_{MN}\bigg(\begin{matrix}a&b\\c&d\end{matrix}\,\bigg|\,u\bigg)\phi_M(a,c|uq^N).\end{equation}
\end{theorem}

\begin{proof}
For simplicity we give the proof for $M=N=2$, but it should be clear that the same argument works in general.
By definition,
\begin{align*}W_{22}\bigg(\begin{matrix}a&b\\c&d\end{matrix}\,\bigg|\,u\bigg)
&=\sum_{ghk}W\bigg(\begin{matrix}a&e\\g&h\end{matrix}\,\bigg|\,u\bigg)W\bigg(\begin{matrix}e&b\\h&f\end{matrix}\,\bigg|\,q^{-1}u\bigg)\\
&\qquad\times W\bigg(\begin{matrix}g&h\\c&k\end{matrix}\,\bigg|\,qu\bigg)W\bigg(\begin{matrix}h&f\\k&d\end{matrix}\,\bigg|\,u\bigg),
  \end{align*}
  where $e$ and $f$ are arbitrary admissible heights. 
Since $|a-g|=|c-g|=1$ we may write
  $$\phi_2(a,c|uq^2)=\phi(a,g|uq)\phi(g,c|u q^2) $$
independently of $g$. Thus, the right-hand side of \eqref{gbm} is
\begin{multline*}\sum_{cghk}W\bigg(\begin{matrix}a&e\\g&h\end{matrix}\,\bigg|\,u\bigg)W\bigg(\begin{matrix}e&b\\h&f\end{matrix}\,\bigg|\,q^{-1}u\bigg)\\
\times W\bigg(\begin{matrix}g&h\\c&k\end{matrix}\,\bigg|\,qu\bigg)W\bigg(\begin{matrix}h&f\\k&d\end{matrix}\,\bigg|\,u\bigg)\phi(a,g|qu)\phi(g,c| q^2u). \end{multline*}
Computing this sum by repeated application  of \eqref{gbe}, first for the sum over $c$, then for $g$ and $k$ and finally for $h$, we eventually arrive at
$$\phi(b,f|q^{-1}u)\phi(f,d|u)=\phi_2(b,d|u).$$
\end{proof}

We now observe that \eqref{gbm} is a special case of the expansion \eqref{ere}. Explicitly, 
\begin{multline*} W_{MN}\bigg(\begin{matrix}a&b\\c&d\end{matrix}\,\bigg|\,u\bigg)=q^{\frac 14(a^2+d^2-b^2-c^2-2MN)} \\
\times R_{(M+d-b)/2}^{(M+c-a)/2}(q^{\frac{1-M+b}2}\sqrt u,q^{\frac{1-M-b}2}\sqrt u,q^{\frac{1-M+N+a}2}\sqrt u,q^{\frac{1-M+N-a}2}\sqrt u;M).
\end{multline*}
Thus, it follows from \eqref{rkl} that $W_{MN}$ can be written as an elliptic hypergeometric sum.

Finally, we mention  that the fused Boltzmann weights satisfy the Yang--Baxter equation in the form
\begin{multline}\label{fyb}\sum_x W_{MN}\bigg(\begin{matrix}a&b\\x&c\end{matrix}\,\bigg|\,u/v\bigg)
W_{MP}\bigg(\begin{matrix}f&a\\e&x\end{matrix}\,\bigg|\,u/w\bigg)
W_{NP}\bigg(\begin{matrix}e&x\\d&c\end{matrix}\,\bigg|\,v/w\bigg)\\
=\sum_x W_{NP}\bigg(\begin{matrix}f&a\\x&b\end{matrix}\,\bigg|\,v/w\bigg)
W_{MP}\bigg(\begin{matrix}x&b\\d&c\end{matrix}\,\bigg|\,u/w\bigg)
W_{MN}\bigg(\begin{matrix}f&x\\e&d\end{matrix}\,\bigg|\,u/v\bigg),\end{multline}
which is quite non-trivial when viewed as a hypergeometric identity. 
We will prove \eqref{fyb}  for $M=N=P=2$, but it will be clear that the  argument works in general. To this end,
we start with the partition functions defined by the pictures
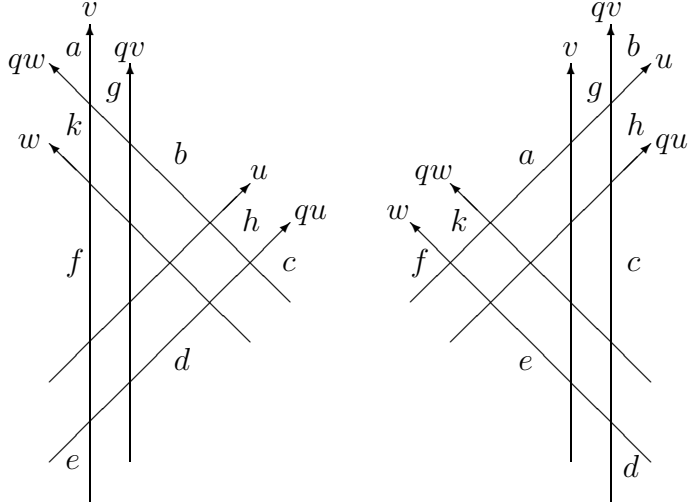
\setlength{\unitlength}{1.5pt}
$$
 \begin{picture}(160,130)
\put(120,40){\vector(1,1){50}}
\put(110,50){\vector(1,1){60}}
\put(170,10){\vector(-1,1){60}}
\put(170,30){\vector(-1,1){50}}
\put(150,10){\vector(0,1){100}}
\put(160,0){\vector(0,1){120}}
\put(70,40){\vector(-1,1){50}}
\put(80,50){\vector(-1,1){60}}
\put(20,10){\vector(1,1){60}}
\put(20,30){\vector(1,1){50}}
\put(40,10){\vector(0,1){100}}
\put(30,0){\vector(0,1){120}}
\put(24,112){$a$}
\put(34,102){$g$}
\put(68,68){$h$}
\put(24,58){$f$}
\put(24,8){$e$}
\put(78,58){$c$}
\put(51,33){$d$}
\put(51,85){$b$}
\put(164,112){$b$}
\put(154,102){$g$}
\put(164,92){$h$}
\put(164,58){$c$}
\put(163,6){$d$}
\put(110,58){$f$}
\put(137,33){$e$}
\put(137,85){$a$}
\put(120,68){$k$}
\put(12,89){$w$}
\put(9.5,109){$qw$}
\put(28,122){$v$}
\put(36,112){$qv$}
\put(70,81){$u$}
\put(81,71){$qu$}
\put(171,89){$qu$}
\put(171,109){$u$}
\put(155,122){$qv$}
\put(148,112){$v$}
\put(111,81){$qw$}
\put(104,71){$w$}
\put(24,92){$k$}
\end{picture}
$$
As usual, we sum over all admissible heights for the empty slots.
We claim that these two partition functions are equal. 
Indeed, using the  Yang--Baxter equation we can pull the lines labelled $v$ and $qv$ through the  crossings
of the other four lines, thus passing between the two pictures without changing the partition function. Note that this one-sentence pictorial (though rigorous) argument corresponds to an eight-fold application of  the identity \eqref{cyb} and would thus look
quite daunting if written out with explicit formulas. Let us now write the partition function on the left as
\setlength{\unitlength}{2pt}
$$\begin{picture}(130,40)
\put(-10,15){$\displaystyle\sum_{x,y,z,t,\dots}$}
\put(20,0){\vector(0,1){30}}
\put(30,0){\vector(0,1){30}}
\put(60,0){\vector(0,1){30}}
\put(70,0){\vector(0,1){30}}
\put(100,0){\vector(0,1){30}}
\put(110,0){\vector(0,1){30}}
\put(10,10){\vector(1,0){30}}
\put(10,20){\vector(1,0){30}}
\put(50,10){\vector(1,0){30}}
\put(50,20){\vector(1,0){30}}
\put(90,10){\vector(1,0){30}}
\put(90,20){\vector(1,0){30}}
\put(14,4){$e$}
\put(34,4){$d$}
\put(54,4){$d$}
\put(74,4){$c$}
\put(94,4){$x$}
\put(104,4){$z$}
\put(114,4){$b$}
\put(34,14){$t$}
\put(54,14){$t$}
\put(74,14){$h$}
\put(94,14){$y$}
\put(114,14){$g$}
\put(14,24){$f$}
\put(24,24){$y$}
\put(34,24){$x$}
\put(54,24){$x$}
\put(64,24){$z$}
\put(74,24){$b$}
\put(94,24){$f$}
\put(114,24){$a$}
\put(104,24){$k$}
\put(18.5,32){$v$}
\put(27.5,32){$qv$}
\put(58.5,32){$w$}
\put(67,32){$qw$}
\put(98.5,32){$w$}
\put(107,32){$qw$}
\put(41,8){$qu$}
\put(41,18){$u$}
\put(81,8){$qu$}
\put(81,18){$u$}
\put(121,8){$qv$}
\put(121,18){$v$}
\end{picture}
 $$
 Here, we only indicate the summation variables that are shared by two factors; each empty slot carries an additional independent summation variable.
 For fixed $t$, $x$ and $y$, the first factor is of the form \eqref{if}. Thus, after summing over the empty slots, it is independent of $y$ and $t$ and equal to
 $$W_{22}\bigg(\begin{matrix}f&x\\e&d\end{matrix}\,\bigg|\,u/v\bigg). $$
 As $t$ now only appears in the second factor, we can sum over the empty slots together with $t$ and obtain
 $$W_{22}\bigg(\begin{matrix}x&b\\d&c\end{matrix}\,\bigg|\,u/w\bigg). $$
Finally, we sum over $y$, $z$ and the empty slot in the third factor and arrive at the right-hand side of \eqref{fyb}.
The left-hand side is obtained in the same way.

\begin{problem}\label{usp}
Using \eqref{bws} and the fact that $\phi(a,c|u)=\phi(c,a|q/u)$, derive the unitarity of the elliptic SOS model from Proposition \ref{gbp}.
\end{problem}

\begin{problem}
Prove a unitarity relation for fused Boltzmann weights and verify that it leads to a special case of the biorthogonality relations of \S \ref{brs}.
\end{problem}

 \end{document}